\documentclass[12pt]
{amsart}
\usepackage{fancyhdr,url}

\usepackage{amsmath, amsthm, amssymb, tikz, hyperref}
\usepackage{amsfonts}
\usepackage[a4paper,
margin=2.5cm,
marginparwidth=1.75cm]{geometry}

\newcommand{\bbz}{\mathbb{Z}}
\newcommand{\bbr}{\mathbb{R}}

\newcommand{\FF}{\mathbb{F}}
\renewcommand\S{{\mathcal S}}

\newcommand\X{{\mathcal X}}

\DeclareMathOperator\Des{Des}

\DeclareMathOperator\maj{maj} 
 
\DeclareMathOperator\Reg{Reg}

\newcommand{\la}{{\lambda}}
\newcommand{\al}{{\alpha}}

\newcommand{\si}{{\sigma}}
\newcommand{\om}{{\omega}}
\newcommand{\Lie}{\operatorname{Lie}}
\newcommand{\Ind}{\operatorname{Ind}}
\newcommand{\Pl}{M_{\rm Planch}}
\newcommand{\Prob}{{\rm Prob}}

\def\magenta#1{{\color{magenta}{#1}}}

\theoremstyle{plain}
\newtheorem{theorem}{Theorem}[section]
\newtheorem{proposition}[theorem]{Proposition}
\newtheorem{lemma}[theorem]{Lemma}
\newtheorem{corollary}[theorem]{Corollary}
\newtheorem{conjecture}[theorem]{Conjecture}

\newtheorem{problem}[theorem]{Problem}

\theoremstyle{definition}

\newtheorem{definition}[theorem]{Definition}

\theoremstyle{remark}

\newtheorem{remark}[theorem]{Remark}

\newcommand{\id}{{\operatorname{id}}}

\def\CC{{\mathbb C}}

\newcommand{\ve}{\varepsilon}

\newcommand{\ch}{{\operatorname{ch}}}
\newcommand{\fix}{{\operatorname{fix}}}

\newcommand{\SYT}{{\operatorname{SYT}}}

\newcommand{\todo}[1]{\vspace{5 mm}\par \noindent
	\marginpar{\textsc{ToDo}} \framebox{\begin{minipage}[c]{0.95
				\textwidth}
			#1 \end{minipage}}\vspace{5 mm}\par}

\title{Asymptotics of higher Lie characters}

\author[R.\ M.\ Adin]{Ron M.\ Adin}
\address{Department of Mathematics, Bar-Ilan University, Ramat-Gan 52900, Israel}
\email{radin@math.biu.ac.il}

\author[Y.\ Roichman]{Yuval Roichman}
\address{Department of Mathematics, Bar-Ilan University, Ramat-Gan 52900, Israel}
\email{yuvalr@math.biu.ac.il}

\author[N.\ Tsilevich]{Natalia Tsilevich}
\address{Department of Applied Mathematics, Braude College of Engineering, Karmiel 2161002, Israel} 
\email{natalia.tsilevich@gmail.com}

\date{September 16, 2025}

\thanks{Part of this research was done while N.~Tsilevich was hosted by  Bar-Ilan University in the framework of the Shapira program.}

\keywords{Asymptotic representation theory, symmetric group, higher Lie characters, conjugacy representation,  Thrall's problem}

\begin{document}

\begin{abstract} 
    Higher Lie characters form a distinguished family of symmetric group characters, which appear in many areas of algebra and combinatorics. An old open problem of Thrall is to decompose them into irreducibles. 
    We propose a novel asymptotic approach to this problem, showing that many families of 
    higher Lie characters 
    tend to be proportional to the regular character. 
    In particular, a random higher Lie character 
    tends, in probability, to be proportional to the regular character.
\end{abstract}

\maketitle

\tableofcontents

\section{Introduction}

\subsection{Higher Lie characters and Thrall's problem}

Higher Lie characters form a distinguished family of symmetric group characters. 
Their study can be traced back to Schur~\cite{Schur}, and then to  papers in Lie theory by Thrall~\cite{Thrall} 
and others. 
The important special case of Lie characters was thoroughly studied by Klyachko~\cite{Klyachko}, Kra\'skiewicz and Weyman~\cite{KraskiewiczWeyman}, and Garsia~\cite{Garsia}, who pointed out 
their significant role in combinatorics. 
A new era in the study of higher Lie characters began when Gessel and Reutenauer~\cite{GR93} presented them by means of quasisymmetric functions, leading to surprising applications in various areas of algebra and combinatorics: Lyndon words~\cite[Ch.~7]{Reutenauer}, \cite[Exercise 7.89]{Stanley_ECII}, root enumeration~\cite[Ex.~7.69]{Stanley_ECII}, construction of Gelfand models~\cite{Inglis}, permutation statistics 
and card shuffling~\cite{Saliola}. 
Strong connections were found to the Whitney homology of the partition lattice~\cite{Stanley1982},  
the cohomology of configuration spaces~\cite{Hersh}, 
and Varchenko--Gelfand and Orlik--Solomon algebras~\cite{Almousa}. 




For a partition $\lambda\vdash n$,
denote by $\chi^\lambda$ and by $\psi^\lambda$ the irreducible character and the higher Lie character indexed by $\lambda$, respectively. 

\begin{problem} {\rm (Thrall~\cite{Thrall})}
Find a combinatorial interpretation of the inner products 
$c_\nu^{\lambda}:=\langle \psi^\lambda,\chi^\nu\rangle$.
\end{problem}

This 83 years old open problem of Thrall 
has received significant attention over the years; 
see, e.g., \cite{Brandt, Klyachko, KraskiewiczWeyman, Garsia, GR93, Reutenauer, Sundaram94, Schocker}. 
Solutions are known for distinguished families of partitions.  
The partitions $\lambda = (2^k 1^{n-2k})$ 
yield $c_\nu^\lambda \in \{0,1\}$, an important multiplicity-free property that was applied to construct Gelfand models~\cite{Inglis}. 
The case $\lambda = (n)$ was solved by Kra\'{s}kiewicz and Weyman~\cite{KraskiewiczWeyman}, using the major index statistic; see Theorem~\ref{thm:KW} below. 
The best general result so far is Schocker's expansion~\cite[Theorem~3.1]{Schocker}, which 
however involves signs and rational coefficients.
Thrall's problem was approached from different viewpoints in the last decade; 
see, e.g., \cite{Hersh, Reiner-Banff, AS_Thrall, Sundaram18, 
AHR, Lim, Pak, Armon, Amendola, Commins}.  



In this paper we propose a novel asymptotic approach to Thrall's problem. 
Instead of focusing on the multiplicities $\langle \psi^\la,\chi^\nu \rangle$ for fixed finite partitions $\la$ and $\nu$, we study their asymptotic behavior as both partitions grow. 
It turns out that, for many families of partitions $\la$,  
the higher Lie character $\psi^\la$ 
tends to be proportional 
to the regular character.





\subsection{Asymptotic proportionality of characters}\label{sec:tendtoregular}


There are various ways to measure how asymptotically close two character sequences $\rho = (\rho_n)$ and $\rho' = (\rho'_n)$ are.
Our choice is
to compare multiplicities 
of irreducible characters,
$\langle \rho_n,\chi^{\nu^{(n)}} \rangle$ and $\langle \rho'_n,\chi^{\nu^{(n)}} \rangle$,
where the sequence of diagrams $\nu = (\nu^{(n)})$ is typical with respect to the Plancherel measure; see Subsection~\ref{sec:Plancherel} for more details about this measure. 
%
%
The choice of Plancherel-typical ``test characters'' $\chi^\nu$ is based on the following consideration.
The Plancherel probability of a 
diagram $\nu$ is proportional to the dimension of the isotypic component of the regular representation $\Reg_n$ corresponding to $\chi^\nu$, which is $(\dim \chi^\nu)^2$. Thus, from a representation-theoretic point of view, it is natural to assume that a ``randomly chosen'' irreducible character $\chi^{\nu}$ is distributed according to the Plancherel measure.
We therefore introduce the following definition.

\begin{definition}
    Two character sequences $\rho = (\rho_n)$ and $\rho' = (\rho'_n)$, where $\rho_n$ and $\rho'_n$ are characters of $\S_n$, are said to be {\em asymptotically proportional} if 
    \begin{equation*}
        \frac{\langle \rho_n,\chi^{\nu^{(n)}} \rangle}{\dim\rho_n}
        \sim \frac{\langle \rho'_n,\chi^{\nu^{(n)}} \rangle}{\dim\rho'_n}
        \quad \text{as} \quad n\to\infty
    \end{equation*}
    for a Plancherel-typical sequence of Young diagrams $\nu=(\nu^{(1)},\nu^{(2)},\ldots)$.
\end{definition}

In particular, if $\rho'_n = \Reg_n$, the regular character of $\S_n$, then $\dim\Reg_n = n!$ and $\langle \Reg_n,\chi^{\nu^{(n)}}\rangle = f^{\nu^{(n)}}$.

\begin{definition}\label{def:regularity}
    A sequence of characters $\rho = (\rho_n)$ {\em tends to be regular} if it satisfies
    \begin{equation}\label{tends_to_regular}
        \frac{\langle \rho_n,\chi^{\nu^{(n)}} \rangle}{\dim\rho_n}
        \sim \frac{f^{\nu^{(n)}}}{n!}
        \quad \text{as} \quad n\to\infty
    \end{equation}
    for Plancherel-typical $\nu = (\nu^{(n)})$.   
  
\end{definition}






\subsection{Existing asymptotic results}




  



The sum of all higher Lie characters of $\S_n$, each with multiplicity $1$, is the regular character $\Reg_n$; see Theorem~\ref{t:sum_of_hLc} for a formulation in terms of symmetric functions. 
Coarse versions of Thrall's problem (see~\cite{Reiner-Banff}) aim to describe the decomposition into irreducibles of other (large) sums of higher Lie characters.
For example, the {\em derangement character} ${\mathcal D}_n$ of the symmetric group $\S_n$  is the sum of higher Lie characters over all partitions with no parts of size~$1$. Thus, the dimension of ${\mathcal D}_n$ is equal to the number $d_n$ of derangements (fixed-point-free permutations) in $\S_n$. For more about the derangement character see, e.g., \cite{DesarmenienWachs, ReinerWebb, AAER}.
Following a combinatorial interpretation by D\'esarm\'enien and Wachs~\cite{DesarmenienWachs}, 
Regev~\cite{Regev99} presented the following remarkable asymptotic result. 


\begin{theorem}\label{thm:Regev}{\rm (Regev~\cite{Regev99})}
    For Plancherel-typical $\nu$,  
    \[
        \langle {\mathcal D}_n,\chi^{\nu} \rangle
        \sim \frac{f^{\nu}}{e} 
        \quad \text{as} \quad n\to \infty.
    \]
\end{theorem}

Note that $\lim_{n \to \infty} d_n / n! = 1/e$, so that Regev's result can be written as
\[
    \frac{\langle {\mathcal D}_n,\chi^{\nu} \rangle}{\dim {\mathcal D}_n}
    \sim \frac{f^{\nu}}{n!}
    \quad \text{as} \quad n\to\infty
\]
for Plancherel-typical $\nu$.
In other words, ${\mathcal D}_n$ tends to be regular in the sense of Definition~\ref{def:regularity}.



\medskip

The action of $\S_n$ on itself by conjugation is closely related to higher Lie characters, as explained in Subsection~\ref{sec:conjugacy} below; denote by ${\mathcal C}_n$ the correspoding character.
The asymptotics of ${\mathcal C}_n$
was first studied in~\cite{AF}, where it was shown that it tends, in a certain sense, to the regular character. 
A well known problem (see, e.g., \cite{Heide}) is to decompose 
${\mathcal C}_n$ into irreducible characters. 
The following result follows from~\cite[Theorem 2.1]{R97}. 

\begin{theorem}\label{thm:Roichman}
    For Plancherel-typical $\nu$, 
\[
    \frac{\langle {\mathcal C}_n,\chi^{\nu} \rangle}{\dim {\mathcal C}_n}
    \sim \frac{f^{\nu}}{n!}
    \quad \text{as} \quad n\to\infty.
\]
\end{theorem}
Thus, ${\mathcal C}_n$ also tends to be regular in the sense of Definition~\ref{def:regularity}.






\subsection{Main results}\label{sec:main_results}

In this paper we consider
Thrall's original problem (rather than its coarse versions) in an asymptotic setting; namely, we study the multiplicities of irreducibles in single higher Lie characters (rather than sums of characters), indexed by diagrams of growing sizes.  

Our first main result deals with sequences of higher Lie characters indexed by rectangular diagrams. 
Note that Thrall's problem may be reduced, via the Littlewood--Richardson rule, to the special case of a rectangular diagram; see, e.g., \cite{Reiner-Banff} and references therein.  

\begin{theorem}\label{main:rectangular}
Any sequence of higher Lie characters, indexed by rectangular diagrams $\la = (m^k)\vdash n$ with
$m \to \infty$ and $k = o(m)$, tends to be regular. 
\end{theorem}

For a precise definition of the notion {\em tends to be regular} see Definition~\ref{def:regularity} below. For  
a stronger version of this result see Theorem~\ref{thm:rectangular} below.



The second main result of this paper deals with the asymptotics of a random higher Lie character. 
As mentioned above,
the sum of higher Lie characters $\psi^\la$ over all partitions $\la \vdash n$ is the regular character $\Reg_n$. Also, $\dim\psi^\la=|C_\la|$, where $C_\la$ is the conjugacy class consisting of all permutations in $\S_n$ with cycle type $\la$. Thus, exactly the same logic as in the choice of Plancherel measure for irreducible characters $\chi^\nu$ suggests that 
a random higher Lie character $\psi^\la$ ($\la \vdash n$) be chosen with probability  $|C_\la|/n!$. 



\begin{theorem}\label{thm:main} 
     A random higher Lie character $\psi^{\la_n}$, where $\la_n \vdash n \to \infty$, 
    tends in probability to be regular.
\end{theorem}

For stronger quantitative versions, see Theorems~\ref{thm:main1} and~\ref{thm:main_general} below. 


To prove Theorem~\ref{thm:main} we analyze, first, diagrams with distinct row lengths, and show that under mild conditions they tend to be regular (Theorem~\ref{thm:distinct} below).  
Another important tool in the proof of Theorem~\ref{thm:main} is the Gluing Lemma (Lemma~\ref{l:gluing} below),
which allows us to glue relatively small diagrams to 
large diagrams 
and preserve the tending-to-regularity property. 
An immediate consequence of the Gluing Lemma is a sufficient condition for higher Lie characters of 
hook shapes to tend to be regular (Corollary~\ref{cor:hooks}).

\begin{remark}
    Theorems~\ref{main:rectangular} and~\ref{thm:main}, as well as all other results in this paper, hold also when higher Lie characters are replaced by any other characters induced from linear characters of centralizers, in particular for the characters of the conjugacy action of $\S_n$ on its conjugacy classes;
    see Theorem~\ref{thm:conjugacy} and Corollary~\ref{cor:conjugacy} below. 
\end{remark}

For the benefit of the reader, here is a list of other results in this paper: 
Theorem~\ref{prop:Lie} (one-row diagrams), 
Theorem~\ref{thm:distinct} (distinct row lengths),
Lemma~\ref{l:gluing} (Gluing Lemma),
Corollary~\ref{cor:hooks} (hook shapes),
Corollary~\ref{thm:main_cor1} (good bounds on the error term or, alternatively, the error probability in Theorem~\ref{thm:main_general}),
Proposition~\ref{lem:virtual} 
and Conjecture~\ref{conj:virtual} (virtual permutations),
Theorem~\ref{thm:conjugacy} (conjugacy characters),
and Corollary~\ref{cor:conjugacy} ($\S_n$-actions on cohomology).

The rest of the paper is organized as follows. 
In Section~\ref{sec:preliminaries} we 
give some necessary background,
and introduce the concept of characters tending to be regular. 
In Section~\ref{sec:1row} we refine a result of Swanson regarding the classical (one-row) Lie character.
In Section~\ref{sec:rectangular} we prove our first main result, Theorem~\ref{thm:rectangular} (a strong version of Theorem~\ref{main:rectangular}), concerning rectangular diagrams.
Sections~\ref{sec:distinct} and~\ref{sec:gluing} contain two useful results, regarding diagrams with distinct parts and the Gluing Lemma. 
These results are used in Section~\ref{sec:asymp} for the proof of Theorem~\ref{thm:main_general}, which implies Theorem~\ref{thm:main}.
Variants of the main results appear in Section~\ref{sec:variations}, and Section~\ref{sec:questions} concludes with some further questions.

\section{Preliminaries}\label{sec:preliminaries}

\subsection{Notation}

We identify partitions with the corresponding Young (or Ferrers) diagrams.
For a 
partition
$\la = (\la_1, \la_2, \dots) = (1^{m_1} 2^{m_2} \cdots) \vdash n$,
\begin{itemize}
    \item 
    $\la_1$ and $\la'_1$ are the lengths of the first row and the first column of $\la$, respectively;
    \item 
    $\chi^\la$ is the irreducible character of $\S_n$ corresponding to $\la$;
    \item 
    $f^\la=\dim\chi^\la$;
    \item 
    $C_\la$ is the conjugacy class in $\S_n$ consisting of all the permutations with cycle type $\la$;
    \item 
    $z_\la:=\prod_ii^{m_i}m_i!$, so that $|C_\la| = n! / z_\la$.
\end{itemize}
 
By $e_\la, h_\la, p_\la, s_\la$ we denote the elementary, complete homogeneous, power sum, and Schur symmetric functions, respectively.



\subsection{Higher Lie characters}\label{sec:hLc}

For $\la=(i^{m_i})\vdash n$, the higher Lie character $\psi^\la$ can be constructed as follows. 
Denote by $Z_\la$ the centralizer in $\S_n$ of a permutation of cycle type $\la$. Then
\[
    Z_\la
    \simeq \prod_{i=1}^n\S_{m_i}[\bbz_i],
\]
where $\bbz_i$ is the cyclic subgroup of $\S_n$ generated by a cycle of length $i$, and $\S_{m_i}[\bbz_i]$ is the wreath product of $\bbz_i$ by $\S_{m_i}$. For each $i$, let $\zeta_i$ be a primitive irreducible character of $\bbz_i$. Then
\begin{equation}\label{eq:wreathproduct}
     \psi^\la
     = \Ind_{Z_\la}^{\S_n} \left(\id_{m_1}[\zeta_1] \otimes \id_{m_2}[\zeta_2] \otimes \dots \otimes \id_{m_n}[\zeta_n] \right),
\end{equation}
where $\id_{m_i}[\zeta_i]$ is the wreath product of $\zeta_i$ by the trivial character $\id_{m_i}$ of $\S_{m_i}$.

In particular, the character $\psi^{(i)}$ corresponding to the one-row 
diagram 
$\la=(i)$ is the much-studied ordinary Lie character corresponding to the action of $\S_n$ on the multilinear component of the free Lie algebra with $n$ generators.

In what follows, we also use the following formula for higher Lie characters (\cite{Thrall}):
\begin{equation}\label{eq:hlcformula}
\psi^{\la}=\Ind_{\S_{m_1}[\S_1]\times\S_{m_2}[\S_2]\times\ldots}^{\S_n}
\left(\id_{m_1}[\psi^{(1)}]\otimes\id_{m_2}[\psi^{(2)}]\otimes\ldots\right),
\end{equation}
where $\S_{m_i}[\S_i]$ is the wreath product of $\S_i$ by $\S_{m_i}$ and $\id_{m_i}[\psi^{(i)}]$ is the wreath product of the character $\psi^{(i)}$ of $\S_i$ by the trivial character $\id_{m_i}$ of $\S_{m_i}$.

Let $L_\la$ be the Frobenius image of the higher Lie character $\psi^\la$ in the algebra of symmetric functions. 
These functions are sometimes called {\em Lyndon symmetric functions}.  
For convenience, denote
$\Lie_n := L_{(n)}$, the Frobenius image of the ordinary Lie character.





The following theorem is well known; see, e.g., \cite[Theorem VII]{Thrall}.

\begin{theorem}\label{t:sum_of_hLc}
    For every $n\ge 1$
    \[
        \sum_{\la \vdash n} L_\la 
        = h_1^n 
        = \sum_{\la\vdash n} f^\la s_\la .
    \]
\end{theorem}



\subsection{The Plancherel measure}\label{sec:Plancherel}

The (infinite-dimensional) {\em Plancherel measure} is a distinguished measure on the space of infinite Young tableaux, 
which plays an important role in the Vershik--Kerov asymptotic representation theory \cite{VershikKerovPlanch, VershikKerovChar}. In particular, it corresponds to the character of the regular representation of the infinite symmetric group. 

This measure is defined as follows (see~\cite{VershikKerovPlanch}). 
Denote by ${\mathcal T}_n$ the space of of all sequences 
$\nu=(\nu^{(1)},\ldots, \nu^{(n)})$ of Young diagrams $\nu^{(i)}\vdash i$, where $\nu^{(i)} \subset \nu^{(i+1)}$; this is the same as the space of all standard Young tableaux with $n$ cells.
Consider the probability measure on ${\mathcal T}_n$ defined by
\[
    M_n(\nu)
    = \frac{f^{\nu^{(n)}}}{n!} 
    \qquad (\forall\, \nu\in{\mathcal T}_n).
\]
We have the natural projections $p_n:{\mathcal T}_n\to{\mathcal T}_{n-1}$, forgetting the last element of a sequence, and ${\mathcal T}:=\varprojlim{\mathcal T}_n$ is the space of infinite sequences  $\nu =(\nu^{(1)},\nu^{(2)},\ldots)$ of Young diagrams $\nu^{(n)}\vdash n$, where $\nu^{(n)} \subset \nu^{(n+1)}$.
It is easy to check that the pushforward of $M_n$ by $p_n$ coincides with $M_{n-1}$.

\begin{definition}
The measure $M:=\varprojlim M_n$ on the
space ${\mathcal T}$ of infinite sequences of Young diagrams is called the (infinite-dimensional) {\em Plancherel measure}.  
\end{definition}

The finite-dimensional marginals of $M$ coincide with the finite Plancherel measures: for every $\al \vdash n$,
\[
M \left( \left\{\nu =(\nu^{(1)},\nu^{(2)},\ldots) \in {\mathcal T} \,:\, \nu^{(n)} = \al \right\} \right) = \frac{(f^\al)^2}{n!}.
\]

\begin{remark}\label{rem:Plancherel_is_balanced}
    The famous Vershik--Kerov limit shape theorem~\cite{VershikKerovPlanch} implies that, for 
    a Plancherel-typical
    sequence of Young diagrams $\nu$, the first row and the first column of $\nu^{(n)}$ grow like $2\sqrt n$  as $n\to\infty$.
\end{remark}


\subsection{Tending to be regular}

Recall that, by Definition~\ref{def:regularity},
a sequence of characters $\rho = (\rho_n)$ tends to be regular if 
    \begin{equation}\label{eq:tends_to_regular2}
        \frac{\langle \rho_n,\chi^{\nu^{(n)}} \rangle}{\dim\rho_n}
        \sim \frac{f^{\nu^{(n)}}}{n!}
        \quad \text{as} \quad n\to\infty
    \end{equation}
for Plancherel-typical $\nu = (\nu^{(n)})$.  
However, all of the results in this paper hold under a weaker assumption on the test sequence $\nu$, namely for balanced sequences in the sense of the following definition.

\begin{definition}\label{def:balanced}\cite{DousseFeray}
    A sequence of diagrams $\nu =(\nu^{(1)},\nu^{(2)},\ldots)$ with $\nu^{(n)} \vdash n$ is said to be {\em balanced} if the lengths of the first row and the first column of $\nu^{(n)}$ are both $O(\sqrt n)$.
\end{definition}

By 
Remark~\ref{rem:Plancherel_is_balanced},
a Plancherel-typical sequence of diagrams is balanced, so if~\eqref{eq:tends_to_regular2} holds for balanced sequences $\nu$, then $\rho$ tends to be regular.


Further, if $\rho_n = \psi^{\la^{(n)}}$ is 
the higher Lie character corresponding to a diagram $\la^{(n)} \vdash n$,
then $\dim\psi^{\la^{(n)}} = |C_{\la^{(n)}}| = n!/z_{\la^{(n)}}$.
In this case
\eqref{eq:tends_to_regular2}~takes the form
\[
    \langle \psi^{\la^{(n)}},\chi^{\nu^{(n)}} \rangle
    \sim
    \frac{f^{\nu^{(n)}}}{z_{\la^{(n)}}} 
    \quad \text{as} \quad n\to \infty.
\]
For convenience, we will abbreviate this to 
\begin{equation}\label{claim}
    \langle \psi^\la,\chi^\nu\rangle\sim\frac{f^\nu}{z_\la}
    \qquad\text{or, equivalently,}\qquad
    \langle L_\la,s_\nu\rangle\sim\frac{f^\nu}{z_\la},
\end{equation}
where the second asymptotic equality is a reformulation of the first one 
in terms of (Lyndon and Schur) symmetric functions.

\subsection
{Standard Young tableaux: enumeration and bounds}
\label{sec:useful_results}



In this subsection we quote results regarding 
standard Young tableaux, which will be used in the paper. 
We begin with two theorems which, combined, essentially provide the asymptotics of the ordinary Lie characters $\psi^{(n)}$; see Theorem~\ref{prop:Lie} below.

For a skew shape $\la/\mu$, 
denote by $\SYT(\lambda/\mu)$ the set of standard Young tableaux of shape $\la/\mu$ and let $f^{\la/\mu}:=|\SYT(\la/\mu)|$. 
For $T\in \SYT(\la)$, define the {\em descent set}
\[
    \Des(T)
    := \{i \,:\, i+1 \text{ is in a lower row of } T \text{ than } i\} 
\]
and the {\em major index} 
\[
    \maj(T) := \sum_{i \in \Des(T)} i.
\]


\begin{theorem}{\rm (Kra\'skiewicz--Weyman~\cite{KraskiewiczWeyman}, see also~\cite[Theorem 8.4]{Garsia})}%
\label{thm:KW}
    For every partition $\nu\vdash n$, the multiplicity $\langle \psi^{(n)},\chi^{\nu} \rangle$ is equal to the cardinality of the set
    \[
        \{T \in \SYT(\nu) \,:\, \maj(T) \equiv 1 \pmod n\}.
    \]	
\end{theorem}

\begin{theorem}{\rm (Swanson~\cite[Theorem 1.8]{Swanson})}%
\label{thm:Swanson}
    For every partition $\nu\vdash n\ge 1$ and all $r$,
    \[
        \left| \frac{|\{T\in \SYT(\nu) \,:\, \maj(T) \equiv r \pmod n\} |}{f^\nu}-\frac{1}{n}
        \right|
        \le \frac{2n^{3/2}}{\sqrt {f^\nu}}.
    \]	
\end{theorem}

For further results in this direction see~\cite{
Billey} and references therein. 

The following result is due to Dousse and F\'eray~\cite{DousseFeray}.

\begin{theorem}\label{t:DousseFeray}\rm{\cite[Theorem 8]{DousseFeray}} 
    There exist constants $C'_1$ and $C'_2$ such that the following holds: if $\la \vdash n$ with $\max\{\la_1, \la'_1\} \le \al(n)$, and $\mu \vdash m$ with $\max\{\mu_1, \mu'_1\} \le \beta(m)$, such that $m < C'_1 n / \al(n)$, then
    \[
        m! \cdot \frac{f^{\la/\mu}}{f^\la f^\mu} 
        \le e^{C'_2 m \beta(m) n^{-1} \al(n)}. 
    \]
\end{theorem}

Letting  $\al(n) = D \sqrt n$ and $\beta(m) = m$, and denoting $d := C'_1/D$ and $b := C'_2 D$, we deduce the following.

\begin{corollary}\label{lem:sub_generic}
    For any $D > 0$ there exist constants $d > 0$ and $b > 0$ such that, 
    if $\la \vdash n$ with $\max\{\la_1, \la'_1\} \le D \sqrt n$ and $\mu \vdash m$ with $m < d \sqrt n$, then
    \[
        \frac{f^{\la/\mu}}{f^\la} 
        \le e^{b m^{2} n^{-1/2}} \cdot \frac{f^\mu}{m!}.
    \]
\end{corollary}

In several places in this paper we need a strengthening of the following result of Vershik and Kerov~\cite{VershikKerovChar} about dimensions of characters: 
if $m$ is fixed then, for every $\beta\vdash m$,
\begin{equation}\label{eq:skewdim}
    \frac{f^{\nu/\beta}}{f^\nu} 
    = \frac{f^\beta}{m!}(1+o(1))
    \quad \text{as} \quad n \to \infty
\end{equation}
for Plancherel-typical $\nu \vdash n$.
However, we need~\eqref{eq:skewdim} in the situation where $m$ is not fixed but grows not very fast, and $\nu$ is balanced but not necessarily Plancherel-typical. 
%
The case we need is covered by~\cite[Theorem 7]{DousseFeray}, with $r = 0$, $\al(n) = D \sqrt n$ and $\beta(m) = m$; it follows easily from its proof that the bound is uniform on $\beta$. 

\begin{lemma}\label{lem:strong_Vershik_Kerov}
    If $m = o(n^{1/4})$, $\nu \vdash n$ is balanced, and $\beta \vdash m$ is arbitrary, then~\eqref{eq:skewdim} holds
    uniformly on $\beta$ as $n \to \infty$,
    with error term $O(m^{2}n^{-1/2})$.
\end{lemma}

\section{The classical one-row Lie character}\label{sec:1row}



We begin with the proof of~\eqref{claim} for 
the ordinary Lie character 
$\psi^{(n)}$,
corresponding to the Lyndon symmetric function $\Lie_n = L_{(n)}$. 
In this case \eqref{claim} immediately 
follows from Theorems~\ref{thm:KW} and~\ref{thm:Swanson},  
but 
we want a good estimate of the error term.


\begin{theorem}\label{prop:Lie}
    The sequence of 
    ordinary Lie characters $\psi^{(n)}$ 
    tends to be regular. 
    Furthermore, as $n \to \infty$,
    \[
        \langle \Lie_n,s_\nu \rangle
        \sim \frac{f^\nu}{n} 
    \]
    for balanced (and hence for Plancherel-typical) $\nu \vdash n$. 
    In fact,
    \[
        \langle \Lie_n,s_\nu \rangle
        = \frac{f^\nu}n(1+r_\nu)
    \]
    with the following two bounds on the error term $r_\nu$.
    \begin{enumerate}
    \item[\rm(a)] 
        For every $s > 0$ there exist constants $N > 0$ and $c > 0$ (both depending only on $s$) such that, for every $n \ge N$, if $\nu \vdash n$ with $\nu_1, \nu'_1 \le n-c$ then 
        \begin{equation}\label{swanson}
            |r_\nu| \le n^{-s}.
        \end{equation}
    
    \item[\rm(b)] 
        For balanced $\nu\vdash n$, there exists a constant $C>0$ (depending only on the constant in $O(\sqrt n)$ in Definition~\ref{def:balanced}) such that
        \begin{equation}
            |r_\nu| 
            \le \left(\frac{C}{n}\right)^\frac{n}{4}.
        \end{equation}
    
    \end{enumerate}
\end{theorem}

\begin{remark}
    The number of partitions $\nu \vdash n$ not satisfying the condition in Theorem~\ref{prop:Lie}(a) is bounded by a constant depending on $s$ but not on $n$.    
\end{remark}

\begin{proof}
    (a) 
    By Theorems~\ref{thm:KW} and~\ref{thm:Swanson},
    \[
        \left| \frac{\langle \psi^{(n)},\chi^{\nu} \rangle}{f^\nu}-\frac{1}{n} \right|
        \le \frac{2n^{3/2}}{\sqrt {f^\nu}}.
    \]	
    In other words,
    \[
        |r_\nu| \le \frac{2n^{5/2}}{\sqrt {f^\nu}}.
    \]
    Given a Young diagram $\nu$ and a cell $(i,j)\in\nu$, denote by $h_{i,j}^{\rm op}$ the opposite hook length at $(i,j)$, i.e., $h_{i,j}^{\rm op}:=i+j-1$. Fix $c>0$ to be chosen later, and let $n\ge 2c$.
    Observe that if $\nu_1,\nu'_1\le n-c$, then $i,j\le n-c$ while $ij\le n$, whence
    \[
        i+j-1\le n-c+\frac n{n-c}-1\le n-c+1,
    \] 
    and for the diagonal excess $N(\nu)$ of $\nu$ we obtain 
    \[  
        N(\nu):=|\nu|-\max\limits_{(i,j)\in\nu}h_{i,j}^{\rm op}\ge c-1.
    \]

    Now fix $s>0$ and take $c=2s+7$.
    By~\cite[Corollary 4.13]{Swanson}, 
    \[
        f^\nu
        \ge \frac{1}{c} \binom{n}{c-1}
        \ge 4 n^{2s+5}
    \]
    for sufficiently large~$n$ depending only on $s$. Thus, 
    \[
        |r_\nu| \le \frac{2n^{5/2}}{\sqrt {f^\nu}}
        \le \frac1{n^s}.
    \]
    

    (b) 
    As mentioned above, we always have $|r_\nu|\le 2n^{5/2}(f^\nu)^{-1/2}$. Let us bound $f^\nu$ from below for balanced~$\nu$. By Definition~\ref{def:balanced}, $\nu_1,\nu'_1\le c\sqrt n$ for some constant $c$. Thus, by the hook length formula and Stirling's formula,
    \[
        f^\nu\ge\frac{n!}{(2c\sqrt n)^n}
        \sim \sqrt{2\pi n}\left(\frac{\sqrt n}{2ce}\right)^n
        \ge \left( \frac{\sqrt n}{2ce} \right)^n, 
    \]
    whence
    \[
        |r_\nu|
        \le 2n^{5/2} \left( \frac{2ce}{\sqrt n} \right)^{\frac n2}
        \le \left( \frac{C}{\sqrt n} \right)^\frac n2.
     \qedhere
    \]
\end{proof}

\section{Rectangular diagrams}\label{sec:rectangular}


Our next goal is to prove~\eqref{claim}, with a good error estimate, for 
rectangular diagrams $\la = (m^k)$. 
Here $z_{(m^k)} = m^k k!$.

\begin{theorem}\label{thm:rectangular}
    Any sequence of higher Lie characters indexed by rectangular diagrams $\la=(m^k)$ with $m\to\infty$ and $k = o(m)$ tends to be regular. Furthermore,
    \begin{equation*}
        \langle L_{(m^k)},s_\nu \rangle
        \sim \frac{f^\nu}{m^kk!} 
        \quad \text{as} \quad m \to \infty
    \end{equation*}
    for balanced (and hence for Plancherel-typical) $\nu \vdash mk$. 
    
\end{theorem}





\begin{proof}
    The case $k=1$ is essentially Theorem~\ref{prop:Lie}(a).

    Assume that $k \ge 2$, and let $s > 0$.
    Fix $c$ as in the proof of Theorem~\ref{prop:Lie}(a). We will say that 
    $\la\vdash m$ is $c$-\textit{good} if $\max\{\la_1,\la_1'\}\le m-c$ and 
    $c$-\textit{bad} otherwise.
    Then, for sufficiently large $m$,
    \begin{equation}\label{aux}
        \Lie_m
        = \sum_{\la\vdash m}\frac{f^\la}m(1+r_\la) s_\la
    \end{equation}
    where for $c$-good $\la$, from~\eqref{swanson} we have $|r_\la|\le m^{-s}$.

By an easy generalization of~\cite[(8.8)]{Macdonald}, for any symmetric functions $f_i$ we have
\[
    h_k \left[ \sum f_i \right]
    = s_{(k)} \left[ \sum f_i \right]
    = \sum_{\substack{j_i \ge 0 \\ \sum j_i=k}}\prod_i h_{j_i}[f_i].
\]
For a collection of indices $\{j_\la\}_{\la\vdash m}$, denote 
\[
    r_{\{j_\la\}}
    := \prod_\la(1+r_\la)^{j_\la}-1.
\]
By \eqref{eq:hlcformula} and \eqref{aux},
\begin{eqnarray}
    L_{(m^k)}
    &=& h_k[\Lie_m] 
    = \sum_{\substack{j_\la \ge 0 \\ \sum j_\la=k}} \prod_{\la \vdash m}h_{j_\la} \left[ \frac{f^\la}m(1+r_\la) s_\la \right] 
    = \frac{1}{m^k} \sum_{\substack{j_\la \ge 0 \\ \sum j_\la=k}} \prod_\la(1+r_\la)^{j_\la} \prod_{\la\vdash m} h_{j_\la} \left[ f^\la s_\la \right] \nonumber \\
    &=& \frac{1}{m^k}\sum_{\substack{j_\la \ge 0 \\ \sum j_\la=k}} (1+r_{\{j_\la\}}) \prod_{\la \vdash m} h_{j_\la} \left[ f^\la s_\la \right]
    = \frac{1}{m^k}\sum_{\substack{j_\la \ge 0 \\ \sum j_\la=k}} \prod_{\la \vdash m} h_{j_\la} \left[f^\la s_\la \right] + R \nonumber \\
    &=& \frac{1}{m^k}h_k \left[ \sum_{\la \vdash m} f^\la s_\la \right] + R 
    = \frac{1}{m^k}h_k \left[h_1^m \right] + R, 
\label{eq:R}
\end{eqnarray}
where
\[
    R = \frac{1}{m^k}\sum_{\substack{j_\la \ge 0 \\ \sum j_\la=k}} r_{\{j_\la\}} \prod_{\la\vdash m} h_{j_\la}\left[f^\la s_\la\right].
\]
Now let $R=R_1+R_2$,
where $R_1$ includes all summands for which $j_\la \ne 0$ only for $c$-good $\la$, and $R_2$ includes all summands for which $j_\la\ne0$ for at least one $c$-bad $\la$.

In order to bound $R_1$, note that for $0 \le x \le 0.5$,
\[
    -2x \le \ln(1-x) \le 0 \le \ln(1+x) \le x ,
\]
and therefore, for $-0.5 \le x \le 0.5$,
\[
    |\ln(1+x)| \le 2|x|.
\]
Similarly, for $0 \le x \le \ln 2$,
\[
    -x \le e^{-x} - 1 \le 0 \le e^x - 1 \le 2x,
\]
and therefore, for $-\ln 2 \le x \le \ln 2$,
\[
    |e^x - 1| \le 2|x|.
\]
For each summand of $R_1$ we have $|r_\la| \le m^{-s} \le 0.5$ (for $m$ large enough), so that
\[
    \left| \ln \left( \prod_\la(1+r_\la)^{j_\la} \right) \right|
    \le \sum_\la j_\la |\ln(1+r_\la)|
    \le 2 \sum_\la j_\la |r_\la|
    \le 2k m^{-s}.
\]
Then, as $m \to \infty$, $2k m^{-s} \le \ln 2$ and therefore
\begin{align*}
    |r_{\{j_\la\}}|
    &= \left| \prod_\la(1+r_\la)^{j_\la} - 1 \right|
    = \left| \exp \left( \ln \left( \prod_\la(1+r_\la)^{j_\la} \right) \right) -1 \right|
    \le 2 \left| \ln \left( \prod_\la(1+r_\la)^{j_\la} \right) \right| \\
    &\le 4k m^{-s}
    = o(m^{-s+1}).
\end{align*}

Further, it is well known that the plethysm of two Schur functions is Schur-positive.
Therefore,
\begin{eqnarray}
    |\langle R_1,s_\nu \rangle|
    &\le& o(m^{-s+1})\left\langle \frac{1}{m^k} \sum_{\substack{\sum j_\la=k \\ j_\la=0\text{ for $c$-bad }\la}} \prod_{\la\vdash m} h_{j_\la} \left[f^\la s_\la \right], s_\nu \right\rangle \\
    &\le& o(m^{-s+1}) \left\langle \frac{1}{m^k} \sum_{\sum j_\la=k} \prod_{\la\vdash m} h_{j_\la} \left[f^\la s_\la \right], s_\nu \right\rangle \nonumber \\
    &=& o(m^{-s+1}) \left\langle \frac{1}{m^k}h_k \left[ \sum_{\la\vdash m} f^\la s_\la \right], s_\nu \right\rangle
    = o(m^{-s+1}) \left\langle \frac{1}{m^k} h_k \left[h_1^m \right], s_\nu \right\rangle.
\label{sigma1}
\end{eqnarray}

Now, each summand from $R_2$ contains a factor of the form $h_{j}[s_\la]$ with $c$-bad $\la$, for which either $\la_1\ge m-c$ or $\la'_1\ge m-c$. Then, by \cite[Theorem~5.1]{GutierrezRosas}, the decomposition of $h_j[s_\la]$ in the Schur basis contains only $s_\tau$ with $\tau\supset\la$, hence with either $\tau_1\ge m-c$ or $\tau'_1\ge m-c$.
But then the decomposition of the whole product $\prod_\la h_{j_\la}[s_\la]$ contains only $s_\tau$ with either $\tau_1\ge m-c$ or $\tau'_1\ge m-c$.
However, for balanced $\nu\vdash km$ we have $\nu_1,\nu'_1=O(\sqrt{km})=o(m)$, which means that for sufficiently large $m$ the corresponding inner product vanishes. Thus, $\langle R_2,s_\nu\rangle=0$ for sufficienly large~$m$. 

Combining this conclusion with~\eqref{eq:R} and~\eqref{sigma1}, we obtain
\begin{equation}\label{firststep}
    \langle L_{(m^k)},s_\nu\rangle
    = \left(1+o(m^{-s+1})\right)\frac1{m^k}\langle h_k[h_1^m],s_\nu\rangle.
\end{equation}    
In order to deal with $\langle h_k[h_1^m],s_\nu\rangle$, we
expand $h_k$ in the basis of power sums:        
\[
    h_k[h_1^m]
    = \sum_{\mu\vdash k} \frac{p_\mu[h_1^m]}{z_\mu}
    = \frac{h_1^n}{k!} 
    + \sum_{\mu \ne (1^k)} \frac{p_{m\mu}}{z_\mu},
\]
where $m \mu \vdash mk$ is obtained from $\mu \vdash k$ by repeating each row $m$ times. 
Thus,  recalling that $\langle h_1^n,s_\nu \rangle = f^\nu$
and $\langle p_\tau,s_\nu\rangle = \chi^\nu(\tau)$ is the value of the irreducible character $\chi^\nu$ at permutations with cycle type $\tau$, we obtain
\begin{equation}\label{eq:A}
    \langle h_k[h_1^m],s_\nu \rangle
    = \frac{f^\nu}{k!} 
    + \sum_{\mu \ne (1^k)} \frac{\langle p_{m\mu},s_\nu \rangle}{z_\mu}
    = \frac{f^\nu}{k!}
    + \sum_{\mu \ne (1^k)}
    \frac{\chi^\nu(m\mu)}{z_\mu}
    = \frac{f^\nu}{k!}(1+A),
\end{equation}
with
\[
    A
    = \sum_{\mu \ne (1^k)}
    \frac{k!}{z_\mu}
    \frac{\chi^\nu(m\mu)}{f^\nu}=\sum_{e\ne g\in\S_k}\frac{\chi^\nu(mg)}{f^\nu},
\]
where, for $g \in \S_k$, the cycle type of the permutation $mg\in\S_{mk}$ is obtained from the cycle type of $g$ by repeating each part $m$ times. Now we use Roichman's~\cite{R96} estimate: there exist constants $0 < a < 1$ and $b > 0$ such that for every $\nu\vdash n$ and every $g\in\S_n$,
\[
    \frac{|\chi^\nu(g)|}{f^\nu}
    \le \left[ \max \left( \frac{\nu_1}{n}, \frac{\nu'_1}{n}, a \right) \right]^{b(n-\fix(g))},
\]
where $\fix(g)$ is the number of fixed points of $g$. In our case, since $\nu$ is balanced, $\frac{\nu_1}{mk},\frac{\nu'_1}{mk} = o(1)$; besides, $\fix(mg)=m\,\fix(g)$. Thus, denoting $q:=a^b$ (obviously, $0<q<1$), for sufficiently large $m$, the left-hand side does not exceed
$q^{m(k-\fix(g))}$, so 
\[
    |A|
    \le \sum_{e \ne g \in \S_k} q^{m(k-\fix(g))}
    = \sum_{i=0}^{k-1} D_{k,i} q^{m(k-i)},
\]
where the so-called ``rencontres number'' $D_{k,i}$ is the number of permutations $g\in\S_k$ with $\fix(g)=i$.
Clearly, $D_{k,i} = \binom{k}{i} D_{k-i,0}$, where $D_{k-i,0} = d_{k-i}$ are the usual derangement numbers. Thus
\[
    |A| 
    \le \sum_{i=0}^{k-1} \binom{k}{i} d_{k-i} q^{m(k-i)}
    = \sum_{j=1}^{k} \binom{k}{j} d_{j} q^{mj}
    \le \sum_{j=1}^{k} \binom{k}{j} j! q^{mj}
    \le \sum_{j=1}^{k} k^j q^{mj}
    =\sum_{j=1}^{k} (kq^m)^j.
\]
Observe that $kq^m=o(mq^m)=o(m^{-t})$ for any $t>0$, so $A = o(m^{-t})$. Plugging this into~\eqref{eq:A} and~\eqref{firststep}, we obtain
$$
\langle L_{(m^k)},s_\nu\rangle=\left(1+o(m^{-s+1})\right)\left(1+o(m^{-t})\right)
\frac{f^\nu}{m^kk!},
$$
which completes the proof.
\end{proof}



\section{
Diagrams with distinct row lengths}\label{sec:distinct}

In this section we prove~\eqref{claim}, with a good error estimate, for 
diagrams with distinct row lengths. The main result of this section, Theorem~\ref{thm:distinct}, plays a crucial role in the proof of Theorem~\ref{thm:main}. 

In the following theorem we actually consider a sequence $\mu^{(1)}, \mu^{(2)}, \dots$ of diagrams, but we write only a single $\mu$ for readability.

\begin{theorem}\label{thm:distinct}
    Let $\mu =(m_1, \dots, m_t) \vdash n$ with $m_1 > \dots > m_t$. 
    If $t\le g(n)$ and $m_t\ge f(n)$ for some functions $f,g$ such that $f(n) \to \infty$, $g(n) = O(f(n) / \ln n)$, and $g(n) = o(\sqrt n)$ (as $n \to \infty$), then $\psi^\mu$ tends to be regular. Furthermore,
    \[
        \langle L_{\mu},s_\nu \rangle 
        \sim \frac{f^\nu}{z_{\mu}} 
        \quad \text{as} \quad n\to \infty,
    \]
    for balanced (and hence for Plancherel-typical) $\nu$.
    Moreover, the error term is uniform over $\nu$:
    \[
        \langle L_{\mu}, s_\nu \rangle 
        = \frac{f^\nu}{z_{\mu}} (1+R_\nu)
        \quad \text{where} \quad |R_\nu| = o(f(n)^{-q})
    \]
    for any fixed (arbitrarily large) $q$.
\end{theorem}

\begin{proof}

Since $\mu$ has distinct parts, we have 
\[
    L_\mu
    = \prod_{i=1}^t L_{(m_i)}
    = \prod_{i=1}^t \Lie_{m_i}.
\]
    Using the Hopf algebra structure on $\Lambda$, for any two symmetric functions $f$ and $g$, and Schur function $s_\nu$ we have
    \begin{eqnarray}\label{eq:Hopf}
        \langle fg, s_\nu \rangle
        &=& \langle f \otimes g, \Delta s_\nu \rangle
        = \langle f\otimes g, \sum_{\la \subseteq \nu} s_\la \otimes s_{\nu/\la} \rangle
        = \sum_{\la \subseteq \nu} \langle f, s_\la \rangle \langle g, s_{\nu/\la} \rangle \nonumber 
        \\
        &=& \sum_{\la \subseteq \nu} \langle f, s_\la \rangle \langle g,\sum_{\mu \subseteq \nu} c^\nu_{\la \mu} s_\mu \rangle 
        = \sum_{\la,\mu \subseteq \nu} c^\nu_{\la \mu} \langle f, s_\la \rangle \langle g, s_\mu\rangle.
    \end{eqnarray}
where $c^\nu_{\la \mu}$ are the Littlewood--Richardson coefficients.
By an obvious generalization of this identity
to an arbitrary number of factors,
\[
    \langle L_\mu,s_\nu \rangle
    = \langle \prod_{i=1}^t L_{(m_i)}, s_\nu \rangle
    = \sum_{\substack{\nu_1,\ldots,\nu_t\\ \nu_i \vdash m_i}} c^\nu_{\nu_1,\ldots,\nu_t} \prod_{i=1}^t \langle L_{(m_i)},s_{\nu_i} \rangle,
\]
where  
$c^\nu_{\nu_1,\ldots,\nu_t}$ is a ``multi-Littlewood--Richardson coefficient'' equal to the coefficient of $s_\nu$ in the product $s_{\nu_1} \cdots s_{\nu_t}$.

For a partition $\nu \vdash m$, define an error term $r_\nu$ by
\begin{equation}\label{eq:def_error}
    \langle L_{(m)},s_\nu \rangle
    = \frac{f^\nu}{m} \cdot (1 + r_\nu) .
\end{equation}
We can then write
\[
    \langle L_\mu,s_\nu \rangle
    = \sum_{\substack{\nu_1,\ldots,\nu_t \\ \nu_i \vdash m_i}} c^\nu_{\nu_1,\ldots,\nu_t} \prod_{i=1}^t \frac{f^{\nu_i}}{m_i} (1 + r_{\nu_i}) .
\]
Defining $r_{\nu_1, \dots, \nu_t}$ by
\[
    1 + r_{\nu_1, \dots, \nu_t} 
    := \prod_{i=1}^{t} (1 + r_{\nu_i}) 
\]
and recalling that $z_\mu  = m_1 \cdots m_t$,
this becomes
\[
    \langle L_\mu,s_\nu \rangle
    = \frac{1}{z_\mu} \sum_{\substack{\nu_1,\ldots,\nu_t \\ \nu_i \vdash m_i}} c^\nu_{\nu_1,\ldots,\nu_t} (1 + r_{\nu_1, \dots, \nu_t}) \prod_{i=1}^t f^{\nu_i} .
\]
Recalling (e.g., from~\cite[Corollary 7.12.5]{Stanley_ECII}) that
\[
    h_1^m = \sum_{\nu \vdash m} f^\nu s_\nu \,,
\]
we have
\begin{align*}
    \sum_{\substack{\nu_1,\ldots,\nu_t \\ \nu_i \vdash m_i}} c^\nu_{\nu_1,\ldots,\nu_t} \prod_{i=1}^t f^{\nu_i} 
    &= \sum_{\substack{\nu_1,\ldots,\nu_t \\ \nu_i \vdash m_i}} \langle s_{\nu_1} \cdots s_{\nu_t}, s_\nu \rangle \prod_{i=1}^t f^{\nu_i} 
    = \left\langle \prod_{i=1}^t \sum_{\nu_i \vdash m_i} f^{\nu_i} s_{\nu_i}, s_\nu \right\rangle 
    = \left\langle \prod_{i=1}^t h_1^{m_i}, s_\nu \right\rangle \\
    &= \langle h_1^n, s_\nu \rangle
    = f^\nu .
\end{align*}
Therefore
\[
    \langle L_\mu,s_\nu \rangle
    = \frac{f^\nu}{z_\mu} \cdot (1 + R) ,
\]
where
\begin{equation}\label{eq:total_error}
    R
    := \frac{1}{f^\nu} \sum_{\substack{\nu_1,\ldots,\nu_t \\ \nu_i \vdash m_i}} c^\nu_{\nu_1,\ldots,\nu_t} r_{\nu_1, \dots, \nu_t} \prod_{i=1}^t f^{\nu_i} .
\end{equation}
Our goal is to prove that, 
for $\mu \vdash n$ 
as in the statement of the theorem and $\nu \vdash n$ balanced, 
$R$ tends to zero (as $n$ tends to infinity)
uniformly over $\nu$.
By the definition of $\nu$ being balanced, there exist a constant $a>0$ such that
\begin{equation}\label{eq:balanced}
    \max\{\nu_1,\nu_1'\}\le a\sqrt n\,.   
\end{equation}



Fix an arbitrary real number $s > 1$. 

By Theorem~\ref{prop:Lie}(a), there exist constants $M > 0$ and $c > 0$ (both depending only on~$s$) such that, for every $m \ge M$ and $\nu \vdash m$ with $\max\{\nu_1, \nu'_1\} \le m-c$, the error term~$r_\nu$ in Equation~\eqref{eq:def_error} satisfies
\[
    |r_\nu| \le m^{-s}.
\]
Call a partition $\nu \vdash m$ {\em 
$c$-good} if $\max\{\nu_1,\nu'_1\} \le m-c$, and {\em $c$-bad} otherwise.
Write 
\[
    R = R_1 + R_2 + R_3,
\]
where
\begin{enumerate}

    \item 
    $R_1$ includes the summands in Equation~\eqref{eq:total_error} for which all $\nu_i$ are $c$-good; 

    \item 
    $R_2$ includes the summands for which at least one $\nu_i$ is $c$-bad and has size $|\nu_i| = m_i < d \sqrt n$, with $d$ to be determined later; and

    \item 
    $R_3$ includes all other summands,
    namely: those for which there exists at least one $c$-bad $\nu_i$, and all the $c$-bad $\nu_i$ have size at least $d \sqrt n$.
        
\end{enumerate}

Let us first deal with $R_1$:
\[
    R_1
    = \frac{1}{f^\nu} \sum_{\substack{\nu_1,\ldots,\nu_t \\ \forall i\, \nu_i \vdash m_i \text{ is $c$-good}}} c^\nu_{\nu_1,\ldots,\nu_t} r_{\nu_1, \dots, \nu_t} \prod_{i=1}^t f^{\nu_i} .
\]
As observed in the proof of Theorem~\ref{thm:rectangular}, 
\[
    |x| \le 0.5 
    \quad \Longrightarrow \quad 
    |\ln(1+x)| \le 2|x|
\]
and 
\[
    |x| \le \ln 2 
    \quad \Longrightarrow \quad
    |e^x - 1| \le 2|x|.
\]
As mentioned above, by Theorem~\ref{prop:Lie}(a), 
if $m \ge M$ and $\nu \vdash m$ is $c$-good, then $|r_\nu| \le m^{-s}$.
We can assume that $M \ge 2^{1/s}$. It follows that, if $m_t \ge M$ and $\nu_1, \dots, \nu_t$ are $c$-good, then (for each $1 \le i \le t$) $|r_{\nu_i}| \le m_i^{-s} \le M^{-s} \le 0.5$ and therefore
\[
    \left| \ln \left( \prod_{i = 1}^t (1 + r_{\nu_i}) \right) \right|
    \le \sum_{i = 1}^t |\ln(1+r_{\nu_i})|
    \le 2 \sum_{i = 1}^t |r_{\nu_i}|
    \le 2 \sum_{i = 1}^t m_i^{-s}
    \le 2t m_t^{-s} .
\]
Since $s>1$, $2t m_t^{-s} \le  2g(n)f^{-s}(n) \to 0$. Hence, for $n$ large enough, $2t m_t^{-s} \le \ln 2$ and 
\[
    |r_{\nu_1, \dots, \nu_t}|
    = \left| \exp \left( \ln \left( \prod_{i = 1}^t (1+r_{\nu_i}) \right) \right) - 1 \right|
    \le 2 \left| \ln \left( \prod_{i = 1}^t (1+r_{\nu_i}) \right) \right|
    \le 4t m_t^{-s} .
\]
Thus, as $n \rightarrow \infty$,
\begin{equation}\label{eq:R1}   
    |R_1| 
    = \left| \frac{1}{f^\nu} \sum_{\substack{\nu_1,\ldots,\nu_t \\ \forall i\, \nu_i \vdash m_i \text{ is $c$-good}}} c^\nu_{\nu_1,\ldots,\nu_t} r_{\nu_1, \dots, \nu_t} \prod_{i=1}^t f^{\nu_i} \right|
    \le 4t m_t^{-s} 
    \le 4g(n)f^{-s}(n)
    \rightarrow 0 \,.
\end{equation}


    
    

Let us now deal with $R_2$:
\[
    R_2 
    = \frac{1}{f^\nu} \sum_{\substack{\nu_1,\ldots,\nu_t \\ \forall i\, \nu_i \vdash m_i \\
    \exists i\, \nu_i \text{ is $c$-bad and } m_i < d \sqrt n}} c^\nu_{\nu_1,\ldots,\nu_t} r_{\nu_1, \dots, \nu_t} \prod_{i=1}^t f^{\nu_i} .
\]
Recall (see Theorem~\ref{t:sum_of_hLc})
that the sum of all higher Lie characters is the regular character.
Equation~\eqref{eq:def_error} thus implies that, for any $\nu_i \vdash m_i$, 
 \[
    \frac{f^{\nu_i}}{m_i} (1+r_{\nu_i}) 
    = \langle L_{(m_i)}, s_{\nu_i} \rangle 
    \le \sum_{\mu \vdash m_i} \langle L_\mu, s_{\nu_i} \rangle
    = \langle h_1^{m_i}, s_{\nu_i} \rangle
    = f^{\nu_i},
\]
so that
\[
    0 \le 1 + r_{\nu_i} \le m_i .
\]
Then
\[
    0
    \le 1+ r_{\nu_1, \ldots, \nu_t} 
    = \prod_{i=1}^{t} \left( 1+r_{\nu_i} \right)
    \le \prod_{i=1}^{t} m_i
    \le n^{t}.
\]
Therefore
\begin{align*}
    |R_2| 
    &\le 
    n^{t} \cdot
    \frac{1}{f^\nu} \sum_{\substack{\nu_1,\ldots,\nu_t \\ \forall i\, \nu_i \vdash m_i \\
    \exists i\, \nu_i \text{ is $c$-bad and } m_i < d \sqrt n}} c^\nu_{\nu_1,\ldots,\nu_t} \prod_{i=1}^t f^{\nu_i} \\
    &\le n^{t} \cdot
    \sum_{\substack{i_0 = 1 \\ m_{i_0} < d \sqrt n}}^{t} \,\,\sum_{\substack{\al \vdash m_{i_0} \\ \al \text{ is $c$-bad}}}
    \frac{1}{f^\nu} \sum_{\substack{\nu_1,\ldots,\nu_t \\ \forall i\, \nu_i \vdash m_i \\
    \nu_{i_0} = \al}} c^\nu_{\nu_1,\ldots,\nu_t} \prod_{i=1}^t f^{\nu_i} \\
    &= n^{t} \cdot
    \sum_{\substack{i_0 = 1 \\ m_{i_0} < d \sqrt n}}^{t} \,\,\sum_{\substack{\al \vdash m_{i_0} \\ \al \text{ is $c$-bad}}}
    \frac{f^{\nu/\al} f^\al}{f^\nu} .
\end{align*}
Recall that $\max\{\nu_1,\nu_1'\}\le a\sqrt n$ by~\eqref{eq:balanced}. Then,
by Corollary~\ref{lem:sub_generic},      
there exist constants $d > 0$ and $b > 0$ such that
if $\al \vdash m$ with $m < d \sqrt n$, then
\[
    \frac{f^{\nu/\al}}{f^\nu} 
    \le e^{b m^{2} n^{-1/2}} \cdot \frac{f^\al}{m!}
    < e^{b d m} \cdot \frac{f^\al}{m!} \, .
\]  
It is easy to see that, for any $c$-bad $\al \vdash m$, $f^\al = O(m^c)$. 
It follows that, for a suitable constant $b' > 0$,
\[
    \frac{f^{\nu/\al} f^\al}{f^\nu} 
    \le \frac{e^{b d m}}{m!} \cdot (f^\al)^2
    \le \frac{e^{b d m}\cdot b' m^{2c}}{m!}
    \le\frac{b'e^{(b d+1)m}m^{2c}} {m^m}.
\]
The function
\[
    \ln\left(\frac{e^{(b d+1)m}m^{2c}} {m^m}\right)
    =(bd+1)m+2c\ln m-m\ln m
\]
is monotone decreasing for $m$ large enough.
A $c$-bad $\al$ has at most $c$ cells outside its first row (or column), which means that the total number of $c$-bad $\al \vdash m$ is at most a constant $c'$  independent of $m$. Denote $C=c'b'$.
By our assumptions $t \le g(n)$, $m_t \ge f(n) \to \infty$, and $g(n) = O(f(n) / \ln n)$, we thus obtain
\begin{align}\label{eq:fromR2}
    |R_2| 
    &\le n^{t} \cdot
    \sum_{\substack{i_0 = 1 \\ m_{i_0} < d \sqrt n}}^{t} \,\,\sum_{\substack{\al \vdash m_{i_0} \\ \al \text{ is $c$-bad}}}
    \frac{f^{\nu/\al} f^\al}{f^\nu} \nonumber\\
    &\le n^t \cdot t \cdot C \cdot 
   e^{(b d + 1) f(n) + 2c\ln f(n) - f(n) \ln f(n) } \\
    &\le C \cdot e^{g(n)\ln n+\ln g(n)
    +(b d + 1) f(n) + 2c\ln f(n) - f(n) \ln f(n) } ,\nonumber
\end{align}
so that, as $n \to \infty$,
\begin{equation}\label{eq:R2}
    |R_2| \le e^{-f(n)\ln f(n)(1+o(1))} \to 0 \,.  
\end{equation}



\medskip

\medskip

It remains to deal with $R_3$:

\[
    R_3 
    = \frac{1}{f^\nu} \sum_{\substack{\nu_1,\ldots,\nu_t \\ \forall i\, \nu_i \vdash m_i \\
    \exists i\, \nu_i \text{ is $c$-bad} \\ \forall i\,\, \nu_i \text{ is $c$-bad } \Rightarrow\, m_i \ge d \sqrt n}} c^\nu_{\nu_1,\ldots,\nu_t} r_{\nu_1, \dots, \nu_t} \prod_{i=1}^t f^{\nu_i} .
\]
As in the discussion of $R_2$, we have
\begin{align}\label{eq:expansionR3}
    |R_3| 
    &\le n^t \cdot
    \sum_{\substack{i_0 = 1 \\ m_{i_0} \ge d \sqrt n}}^{t} \,\,\sum_{\substack{\al \vdash m_{i_0} \\ \al \text{ is $c$-bad}}}
    \frac{f^{\nu/\al} f^\al}{f^\nu} .
\end{align}
Here we cannot use Corollary~\ref{lem:sub_generic}, and need a more delicate argument to bound $f^{\nu/\al}/f^\nu$.

Consider a $c$-bad $\al\vdash m_{i_0}$. For simplicity, write $m := m_{i_0}$.
By the $R_3$ assumption, $m \ge d \sqrt n$. 
Assume that $\al_1 = \max \{\al_1, \al'_1\}$; otherwise, exchange rows and columns in the following discussion.
Thus $\al_1 > m - c$.
For sufficiently large $n$, we have $d \sqrt n - c\ge \delta\sqrt n$ for a suitable constant $\delta>0$.
Since $\al$ is contained in $\nu$, we have by~\eqref{eq:balanced}
\begin{equation}\label{eq:alpah_nu_bounds}
    \delta\sqrt n
    \le d \sqrt n - c
    \le m - c
    < \al_1
    \le \nu_1 
    < a \sqrt n. 
\end{equation}

Given a diagram $\nu$, let $\nu^-:=\nu/(\nu_1)$ denote the (ordinary) diagram obtained from $\nu$ by removing the first row.
We can build a standard Young tableau of shape $\nu/\al$ by first choosing the entries in its first row (of length $\nu_1 - \al_1$), and then placing the other entries into  
$\nu^-/\al^-$ (if possible). Hence
\begin{eqnarray}\label{eq:skew}
   \frac{f^{\nu/\al} f^\al}{f^\nu}
   \le \binom{n-m}{\nu_1-\al_1} \frac{f^{\nu^-/\al^-} f^\al}{f^\nu}
   = \binom{n-m}{\nu_1-\al_1} \cdot \frac{f^{\nu^-/\al^-}}{f^{\nu^-}} \cdot \frac{f^{\nu^-}}{f^\nu} \cdot f^\al. 
\end{eqnarray}

We will bound, separately, each of the four factors in the RHS of~\eqref{eq:skew}.
To that end, we will use the following asymptotic formula (see, e.g., 
\cite[(5.12)]{Spencer}): for $k=o(n)$,
\begin{equation}\label{eq:binomial}
    \binom{n}{k}
    \sim \left( \frac{ne}{k} \right)^k (2\pi k)^{-1/2} e^{-\frac{k^2}{2n}(1+o(1))}
    \sim \frac{n^k}{k!} e^{-\frac{k^2}{2n}(1+o(1))}.
\end{equation}

For sufficiently large $n$, we have $\nu_1-\al_1 \le (a-\delta) \sqrt n$. 
Using the first asymptotic equality in~\eqref{eq:binomial} for $k = \nu_1 - \al_1$,
we obtain, for a suitable constant $d_1 > 0$,
\begin{equation}\label{eq:factor1}
   \binom{n-m}{\nu_1-\al_1}
   \le \binom{n}{\nu_1-\al_1}
   \le d_1^{\sqrt n}n^{\nu_1-\al_1}(\nu_1-\al_1)^{-(\nu_1-\al_1)}\, 
\end{equation}
(with the convention that if $\nu_1-\al_1=0$, then $(\nu_1-\al_1)^{-(\nu_1-\al_1)}=1$).

Observe that $\al^-\vdash c'$ with $c' < c$, and $\nu^-$ satisfies
$\nu^-_1,(\nu^-)'_1\le a\sqrt n$.
Hence, by Lemma~\ref{lem:strong_Vershik_Kerov},
\begin{equation}\label{eq:factor2}
    \frac{f^{\nu^-/\al^-}}{f^{\nu^-}}
    \sim \frac{f^{\al^-}}{c'!}
    = O(1).
\end{equation}

By the hook length formula, denoting by $h_{ij}$ the hook length at a cell $(i,j)$ of $\nu$ and observing that $h_{ij} \le \nu_1 + \nu_1' \le 2a \sqrt{n}$, we have
\[
    \frac{f^{\nu^-}}{f^\nu}
    = \frac{(n-\nu_1)!}{\prod_{(i,j)\in\nu^-}h_{ij}}
    \frac{\prod_{(i,j)\in\nu}h_{ij}}{n!}
    =\frac{(n-\nu_1)!\prod_{j=1}^{\nu_1}h_{1j}}{n!}
    \le\frac{(n-\nu_1)!}{n!}(2a\sqrt{n})^{\nu_1}.
\]
By the assumptions on~$\nu$ and the second asymptotic equality in~\eqref{eq:binomial} for $k = \nu_1$, 
\[
    \frac{(n-\nu_1)!}{n!} 
    = \left( \nu_1! \binom{n}{\nu_1} \right)^{-1}
    \sim  n^{-\nu_1} e^{\frac{\nu_1^2}{2n}(1 + o(1))} = n^{-\nu_1} \cdot O(1).
\]
Therefore,
\begin{equation}\label{eq:factor3}
   \frac{f^{\nu^-}}{f^\nu}
   \le O(1) \cdot n^{-\nu_1} \cdot
   (2a\sqrt{n})^{\nu_1}
   \le  d_2^{\sqrt n} n^{-\nu_1/2}.
\end{equation}

Finally, for any $c$-bad $\al \vdash m$, 
\begin{equation}\label{eq:factor4}
    f^\al = O(m^c) = O(n^c).
\end{equation} 

Plugging inequalities~\eqref{eq:factor1}, \eqref{eq:factor2}, \eqref{eq:factor3}, and~\eqref{eq:factor4} into~\eqref{eq:skew}, we obtain
\[
    \frac{f^{\nu/\al} f^\al}{f^\nu}
    \le d_1^{\sqrt n}n^{\nu_1-\al_1}(\nu_1-\al_1)^{-(\nu_1-\al_1)}
    \cdot O(1) 
    \cdot d_2^{\sqrt n} n^{-\nu_1/2}
    \cdot O(n^c)
    \le d_3^{\sqrt n} n^{(\nu_1/2-\al_1)} (\nu_1-\al_1)^{-(\nu_1-\al_1)}\, .
\]
Consider two cases. 
If $2\nu_1/3\le\al_1 \le \nu_1$, then 
$(\nu_1-\al_1)^{-(\nu_1-\al_1)} \le 1$ (by convention, also if $\nu_1 - \al_1 = 0$).
Also $\nu_1/2-\al_1 \le -\nu_1/6 \le - \delta \sqrt n/6$, so that
\[
    \frac{f^{\nu/\al} f^\al}{f^\nu}
    \le d_3^{\sqrt n} n^{-\delta \sqrt n/6}.
\]
On the other hand, if $\al_1<2\nu_1/3$, then $\nu_1-\al_1 > \nu_1/3 \ge \delta \sqrt n/3$, so that
\[
    (\nu_1-\al_1)^{-(\nu_1-\al_1)}
    \le (\delta \sqrt n/3)^{-(\nu_1-\al_1)}
\]
and
\[
    \frac{f^{\nu/\al} f^\al}{f^\nu}
    \le d_3^{\sqrt n}n^{(\nu_1/2-\al_1)} (\delta \sqrt n/3)^{-(\nu_1-\al_1)}
    \le d_4^{\sqrt n} n^{-\al_1/2}
    \le d_4^{\sqrt n} n^{-\delta \sqrt n/2} .
\]
Thus, in both cases, 
\[
    \frac{f^{\nu/\al} f^\al}{f^\nu}
    \le d_5^{\sqrt n}n^{-d_6\sqrt n}
\]
for positive constants $d_5,d_6$.
Using the same reasoning as in~\eqref{eq:fromR2} and the assumption $t \le g(n) = o(\sqrt n)$, formula~\eqref{eq:expansionR3} yields
\[
    |R_3|
    \le n^t\cdot t\cdot C\cdot d_5^{\sqrt n}n^{-d_6\sqrt n}
    \le C\cdot e^{g(n)\ln n + \ln g(n) + d_7\sqrt n-d_6\sqrt n\ln n}
\]
so that, as $n\to\infty$,
\begin{equation}\label{eq:R3}
   |R_3| 
   \le e^{-d_6\sqrt n\ln n(1+o(1))}
   \to 0 \,.
\end{equation}

It can be easily seen from the bounds \eqref{eq:R1},  \eqref{eq:R2}, and~\eqref{eq:R3}
that $R_2,R_3=o(R_1)$. Thus, the total error term
$R=R_1+R_2+R_3=O(R_1)=O(g(n)f^{-s}(n))=o(f^{-(s-1)}(n))$.
\end{proof}


\section{A gluing lemma}\label{sec:gluing}

The following ``Gluing Lemma'' is essential to the proof of Theorem~\ref{thm:main1}.

\begin{lemma}[Gluing Lemma]\label{l:gluing}
    Let $\la\vdash n$ be the disjoint union $\la=\mu\cup\tau$ of diagrams $\mu\vdash n-k$ and $\tau\vdash k$. Assume that 
    $k=o(n^{1/4})$,
    the smallest part of $\mu$ is larger than the largest part of $\tau$, 
    and the character $\psi^\mu$ tends to be regular strongly and uniformly,  
    in the sense that for any balanced $\al\vdash n-k$
\[
    \langle L_\mu,s_\al\rangle=\frac{f^\al}{z_\mu}(1+R_\al) 
    \quad \text{with} \quad |R_\al| < \omega(n-k) 
    \quad \text{as} \quad n \to \infty,
\]
    for some function $\omega(n)=o(1)$.
    Then $\psi^\la$ tends to be regular. Furthermore,
    \[
        \langle L_\la,s_\nu \rangle
        \sim \frac{f^\nu}{z_\mu z_\tau}
        \quad \text{as} \quad n \to \infty,
    \] 
    for balanced (and hence for Plancherel-typical) $\nu\vdash n$. More precisely, 
    \[
        \langle L_\la,s_\nu\rangle=\frac{f^\nu}{z_\mu z_\tau} (1+R_{\la,\nu})
        \quad \text{with} \quad
        |R_{\la,\nu}| \le O(\om(n-k) + k^2n^{-1/2}).
    \]
\end{lemma}

\begin{proof}
Since $\mu$ and $\tau$ have no common part, 
by~\eqref{eq:Hopf} we have
\[
    \langle L_{\mu\cup\tau},s_\nu \rangle
    = \langle L_\mu L_\tau,s_\nu \rangle
    = \sum_{\al \,\vdash n-k,\,\beta \,\vdash k} c^\nu_{\al\beta} \langle L_\mu,s_\al \rangle \langle L_\tau,s_{\beta} \rangle.
\]
The Littlewood--Richardson coefficient $c^\nu_{\al\beta}$ is nonzero only for $\al \subseteq \nu$. Since $\nu$ is balanced, the rows and columns of $\nu$ are $O(\sqrt{n})=O(\sqrt{n-k})$, hence the same is true for $\al \subseteq \nu$, i.e., $\al\vdash n-k$ is balanced.  
By the assumptions of the theorem,
\[
    \langle L_\mu,s_\al \rangle
    = \frac{f^\al}{z_\mu} (1+R_\al)
\]
with $|R_\al| \le \om(n-k) = o(1)$. 
Therefore
\begin{eqnarray*}
    \langle L_{\mu\cup\tau},s_\nu \rangle
    &=& z_\mu^{-1} \sum_{\al \,\vdash n-k,\,\beta \,\vdash k} c^\nu_{\al\beta} f^\al (1+R_\al) \langle L_\tau,s_{\beta} \rangle \nonumber \\
    &=& (1 + R_1) \cdot z_\mu^{-1} \sum_{\al \,\vdash n-k,\,\beta \,\vdash k} c^\nu_{\al\beta} f^\al \langle L_\tau,s_{\beta} \rangle,
\end{eqnarray*} 
where $|R_1| \le \om(n-k)$. 
Recalling that $f^\al=\langle h_1^{n-k},s_\al\rangle$, it follows from 
\eqref{eq:Hopf} that
\begin{eqnarray*}
    \langle L_{\mu\cup\tau},s_\nu \rangle
    &=& (1 + R_1) \cdot z_\mu^{-1} \sum_{\al \,\vdash n-k,\,\beta \,\vdash k} c^\nu_{\al\beta}f^\al\langle L_\tau,s_{\beta} \rangle \nonumber \\
    &=& (1 + R_1) \cdot z_\mu^{-1} \sum_{\al \,\vdash n-k,\,\beta \,\vdash k} c^\nu_{\al\beta}\langle h_1^{n-k},s_\al \rangle \langle L_\tau,s_{\beta} \rangle \nonumber \\
    &=& (1 + R_1) \cdot z_\mu^{-1} \sum_{\beta\,\vdash k} \langle L_\tau,s_\beta \rangle \langle h_1^{n-k},s_{\nu/\beta} \rangle \\
    &=& (1 + R_1) \cdot z_\mu^{-1} \sum_{\beta\,\vdash k} \langle L_\tau,s_\beta \rangle f^{\nu/\beta}.
\end{eqnarray*}
Since $k = o(n^{1/4})$, Lemma~\ref{lem:strong_Vershik_Kerov} yields, for balanced $\nu \vdash n$,
\begin{equation*}
    f^{\nu/\beta} 
    = \frac{f^\nu f^\beta}{k!} (1 + R_{\nu,\beta})    
\end{equation*}
where $|R_{\nu,\beta}| \le \om_2(n,k) = O(k^2n^{-1/2})$.
Therefore
\begin{eqnarray*}
    \langle L_{\mu\cup\tau},s_\nu \rangle
    &=& (1 + R_1) \cdot z_\mu^{-1} \sum_{\beta\,\vdash k} \langle L_\tau,s_\beta \rangle \frac{f^\nu f^\beta}{k!} (1 + R_{\nu,\beta}) \\
    &=& (1 + R_1) (1 + R_2) \cdot z_\mu^{-1} \sum_{\beta\,\vdash k} \langle L_\tau,s_\beta \rangle \frac{f^\nu f^\beta}{k!} ,
\end{eqnarray*}
where $|R_2| \le \om_2(n,k) = O(k^2n^{-1/2})$.
Now 
\begin{equation*}
    \frac{f^\nu}{k!} \sum_{\beta\,\vdash k} \langle L_\tau,s_\beta \rangle f^\beta
    = \frac{f^\nu}{k!} \langle L_{\tau},h_1^k \rangle
    = \frac{f^\nu}{z_\tau},
\end{equation*}
since 
\[
    \langle L_{\tau},h_1^k\rangle=\dim\psi^\tau=|C_\tau|=\frac{k!}{z_\tau}.
\]
We conclude that
\[
    \langle L_{\mu\cup\tau},s_\nu \rangle
    = (1 + R_1) (1+ R_2) \frac{f^\nu}{z_\mu z_\tau} 
    = (1 + R_3) \frac{f^\nu}{z_\mu z_\tau} ,
\]
where 
\[
    |R_3| \le O(\om(n-k) + k^2 n^{-1/2}) = o(1),
\]
as required.
\end{proof}

\begin{corollary}\label{cor:hooks}
Any sequence of higher Lie characters indexed by 
hook shapes $(n-k,1^k)$ with $k=o(n^{1/4})$ tends to be regular.    
\end{corollary}
\begin{proof}
    Immediately follows from Theorem~\ref{prop:Lie} and Lemma~\ref{l:gluing}.
\end{proof}

\section{The asymptotics of a random higher Lie character}\label{sec:asymp}

In this section we consider the asymptotic behavior of the higher Lie character $\psi^\la$ for a {\em random} partition $\la \vdash n$. 
As explained in Subsection~\ref{sec:main_results},
it is natural to use the probability distribution
\[
    \Prob(\la) 
    = \frac{\dim \psi^\la}{\dim \Reg_n}
    = \frac{|C_\la|}{n!}
    \qquad (\forall\,\la \vdash n),
\]
where $C_\la$ is the conjugacy class consisting of all permutations in $\S_n$ with cycle type $\la$.
Thus, $\la=\la(\si)$ can be described as the cycle structure of a permutation $\si$ chosen uniformly at random from $\S_n$. 
This equivalent description will be used in the statements and proofs in this section. 
On the other hand, the irreducible characters $\chi^\nu$ (or, equivalently, the Schur function $s_\nu$) will be taken, as usual, according to the Plancherel distribution.


\begin{definition}\label{def:tends_in_probability}
    A random sequence of characters $\rho = (\rho_n)$, with $\rho_n$ a character of $\S_n$, {\em tends in probability to be regular} if the following holds: 
    if
    \[
        \langle \rho_n,\chi^{\nu_n} \rangle 
        = \frac{\dim \rho_n}{n!} \cdot f^{\nu_n} \cdot (1 + R_n),
    \]
    then, for every $\ve > 0$,
    \[
        \lim_{n \to \infty} \Prob\,(|R_n| > \ve) = 0
    \]
    for Plancherel-typical $\nu=(\nu_1,\nu_2,\ldots)$. 
\end{definition}

The main result of this section is the following.



\begin{theorem}\label{thm:main1}
    If $\si_n$ is chosen uniformly at random in~$\S_n$ then $\psi^{\la(\si_n)}$ tends in probability to be regular. 
    Furthermore, 
    if
    \[
        \langle L_{\la(\si_n)},s_{\nu_n} \rangle 
        = \frac{f^{\nu_n}}{z_{\la(\si_n)}} (1+R_n) 
    \]
    where $\si_n$ is chosen uniformly at random in~$\S_n$, then, for every $\ve >0$,
    \[
        \lim_{n \to \infty} \Prob\,(|R_n| > \ve) = 0
    \]
    for balanced (and hence for Plancherel-typical) $\nu=(\nu_1,\nu_2,\ldots)$. 
\end{theorem}

The first part of Theorem~\ref{thm:main1} is Theorem~\ref{thm:main} from the Introduction.
In fact, we prove the following more general version.

\begin{theorem}\label{thm:main_general}
    Let $f(n)$ and $g(n)$ be two functions satisfying the following conditions as $n\to\infty$:
    $f$~is (eventually) monotone increasing, 
    $g(n) / \ln n\to\infty$, $g(n) = O(f(n) / \ln n)$, 
    and $f(n)g(n)=o(n^{1/4})$. Denote, for any $s>0$,
    \[
        \ve_n := \frac1{f^s(n/2)} + \frac{f^2(n)g^2(n)}{\sqrt n},
        \quad 
        \theta_n
        := \frac{3 \ln n}{g(n)}.
    \]
    If $\si_n$ is chosen uniformly at random in~$\S_n$, and
    \[
        \langle L_{\la(\si_n)},s_{\nu_n} \rangle 
        = \frac{f^{\nu_n}}{z_{\la(\si_n)}} (1+R_n), 
    \]
    then, for some positive constant $C$, 
    \[
        \Prob\,(|R_n| > C \ve_n) <\theta_n
    \]
    for sufficiently large $n$ and for balanced (and hence for Plancherel-typical) $\nu=(\nu_1,\nu_2,\ldots)$. 
\end{theorem}

In order to prove Theorem~\ref{thm:main_general}, 
we need the following lemma.

\begin{lemma}\label{lem:equalcycles}
    The probability for a random permutation, uniformly distributed on~$\S_n$, to have two cycles of the same length $m$ for some $m > M$  is at most $1/M$.
\end{lemma}
\begin{proof}
    Consider the number of permutations $\si\in\S_n$ that have two cycles of a fixed length $m$. To determine such a permutation, we can choose the $2m$ elements that belong to these two cycles, choose how they are distributed between the cycles and how they are organized inside the cycles, and then choose a~permutation of the remaining $n-2m$ elements. Thus, the number in question is at most
    \[
        \binom{n}{2m} \cdot \frac{1}{2} \binom{2m}{m} ((m-1)!)^2\cdot (n-2m)!
        = \frac{n!}{2m^2}.
    \]
    Therefore, the probability for a random permutation to have two cycles of the same length $m$ is at most~$m^{-2}$. It follows that the probability to have two cycles of arbitrary length $m>M$ is at most
    \[
        \sum_{m=M+1}^{\lfloor n/2\rfloor}\frac 1{m^2}
        \le \int_{M}^{\infty}\frac1{x^2}dx 
        = \frac1{M}.
        \qedhere
    \]
\end{proof}

\begin{proof}[Proof of Theorem~\ref{thm:main_general}.] 
    Choose $\si_n$ uniformly at random in $\S_n$, and write its cycle type $\la=\la(\si_n)$ as a disjoint union $\la=\mu\cup\tau$, where $\mu$ is the part of $\la$ consisting of the rows of length $> f(n)$ and $\tau$ is the part consisting of the rows of length $\le f(n)$.
    
    By Lemma~\ref{lem:equalcycles}, for sufficiently large $n$ all rows of $\mu$ are distinct with probability at least $1-\frac{1}{f(n)}$. 
    
    Denote by $c(\si)$ the number of cycles of a permutation $\si$. It is well known (see, e.g., \cite[(6.5)]{Feller}) that for a random permutation $\si_n\in\S_n$ we have 
    $Ec(\si_n)\sim\ln n$ as $n\to\infty$, where by $E$ we denote the expectation of a random variable. 
    Then, by Markov's inequality, for sufficiently large $n$,
    \[
        \operatorname{Prob}(c(\si_n) > g(n))
        \le \frac{Ec(\si_n)}{g(n)}
        < \frac{2 \ln n}{g(n)}.
    \]
    Assume, from now on, that all the rows of $\mu$ are distinct and that $c(\si_n) \le g(n)$.
    For sufficiently large $n$, this happens with probability at least $1 - \frac{1}{f(n)} - \frac{2 \ln n}{g(n)} > 1 - \frac{3 \ln n}{g(n)} = 1-\theta_n$.
    Note that, by assumption, $g(n) \to \infty$ and $f(n) g(n) = o(n^{1/4})$, thus $f(n) = o(n^{1/4})$, hence $g(n) = O(f(n) / \ln n) = o(n^{1/4}) = o(n^{1/2})$.
    Thus all the assumptions of Theorem~\ref{thm:distinct} are satisfied, so that
    \[
        \langle L_{\mu},s_\nu \rangle = \frac{f^\nu}{z_{\mu}}(1+R_n)
        \quad \text{where} \quad
        |R_n| \le \om(n) = o(f(n)^{-q}),
    \] 
    for any fixed (arbitrarily large) $q$.
    
    As for $\tau$, all its rows are of length at most $f(n)$ and
    there are at most $g(n)$ of them. 
    The total size of $\tau$ is $k \le f(n) g(n)$, so by assumption $k = o(n^{1/4})$. 
    The assumptions of the Gluing Lemma (Lemma~\ref{l:gluing}) are satisfied with $\om(n) = o\!\left(f(n)^{-q} \right)$, so
    \[
        \langle L_{\la},s_\nu \rangle 
    = \frac{f^\nu}{z_\la}(1+R'_n)
    \]
    where, by monotonicity of $f$,
    \[
        |R'_n| 
        = O(f(n-k)^{-q} + k^2 n^{-1/2})
        = O(f(n/2)^{-q} + f(n)^2 g(n)^2 n^{-1/2}) \le C \ve_n.
    \]
    This completes the proof of the theorem.
\end{proof}

\begin{proof}[Proof of Theorem~\ref{thm:main1}.] 
It suffices to take any functions $f(n)$ and $g(n)$ satisfying the assumptions of  Theorem~\ref{thm:main1} --- for instance,
$f(n)=\ln^3n$ and $g(n)=\ln^2n$. 
Since
$f(n)\to \infty$ and $f^2(n)g^2(n))=o(\sqrt n)$, we have 
$\ve_n\to 0$. Also, $g(n) / \ln n\to\infty$, thus $\theta_n\to 0$. 
\end{proof}

\begin{remark}
    There is a trade-off between the sizes of the error term $\ve_n$ and the error probability~$\theta_n$. Choosing the functions $f(n)$ and $g(n)$ appropriately, one can optimize whatever one wants. 
\end{remark}

For example, to achieve a good error term $\ve_n$, take $f(n) = n^{\delta'}$ for an arbitrarily small $0<\delta'<1/4$ and $g(n) = 3 (\ln n)^2$. 
On the other hand, to get a good error probability $\theta_n$, take  $f(n)=n^{1/8}$ and $g(n)=n^{1/8}/\ln n$. 
Thus Theorem~\ref{thm:main_general} yields the following.


\begin{corollary}\label{thm:main_cor1}
    Assume that $\si_n$ is chosen uniformly at random in~$\S_n$ and denote
    \[
        \langle L_{\la(\si_n)},s_{\nu_n} \rangle 
        = \frac{f^{\nu_n}}{z_{\la(\si_n)}} (1+R_n). 
    \]
    Then, for arbitrarily small $\delta>0$, 
    \begin{itemize}
        \item[(a)] 
        \[
            \Prob\,(|R_n| > n^{-1/2+\delta}) 
            < 1 / \ln n 
        \]
        \item[(b)]
        and
        \[
            \Prob\,(|R_n| > 1/\ln n) 
            < n^{-1/8 + \delta}
        \]

    \end{itemize}
    for sufficiently large $n$ and for balanced (and hence for Plancherel-typical) $\nu=(\nu_1,\nu_2,\ldots)$. 
\end{corollary}

\section{Variations}\label{sec:variations}

\subsection{Virtual permutations} 

If, considering the behavior of random higher Lie characters, we want to achieve convergence almost surely instead of convergence in probability, we should introduce a probability measure on {\em infinite sequences} of higher Lie characters. A natural way to achieve this is to consider the space of virtual permutations, introduced in~\cite{KOV} in the context of harmonic analysis on the infinite symmetric group.

Let $\pi_{n,n+1}:\S_{n+1} \to \S_{n}$ be the map that, given a permutation $\si\in\S_{n+1}$, deletes the element $n+1$ from its cycle. The  {\it space of virtual permutations} $\X:=\varprojlim S_n$ is the inverse limit of finite symmetric groups $\S_n$ with respect to these maps. In other words, a virtual permutation, i.e., an element of $\X$, is a sequence $\omega =(\si_1,\si_2,\ldots) \in \S_1\times\S_2\times\ldots$ such that $\si_{n+1}$ is obtained from $\si_n$ by either inserting $n+1$ into one of the existing cycles, or adding $n+1$ as a fixed point; we say that such a sequence of permutations is {\it coherent}.
It is easy to see that the sequence of uniform measures on $\S_n$ agrees with the projections $\pi_{n,n+1}$, so the inverse limit of these measures is a probability measure $M$ on $\X$, which is called the {\it Haar measure} on $\X$. By definition, the projection of $M$ to $\S_n$ coincides with the uniform measure.

\begin{proposition}\label{lem:virtual}
    Let $m_n$ be an increasing sequence of positive integers such that $m_n = \Omega(n^{9})$. 
    Then, for almost every (with respect to the Haar measure) virtual permutation $\sigma =(\si_1,\si_2,\ldots) \in \S_1\times\S_2\times\ldots$,  
    \[
        \langle \psi^{\la(\si_{m_n})},\chi^{\nu_{m_n}} \rangle
        \sim \frac{f^{\nu_{m_n}}}{z_{\la(\si_{m_n})}} 
        \quad \text{as} \quad n \to \infty
   \]
   for balanced (and hence for Plancherel-typical) $\nu=(\nu_1,\nu_2,\ldots)$.
\end{proposition}

\begin{proof}
By the definition of the Haar measure on $\X$, for each $n$ the permutation $\si_n$ is uniformly distributed on~$\S_n$. Let
    \[
        \langle \psi^{\la(\si_{m_n})},\chi^{\nu_{m_n}} \rangle 
        = \frac{f^{\nu_{m_n}}}{z_{\la(\si_{m_n})}} (1+R_{m_n}(\sigma)). 
    \]
Fix an arbitrary $\ve>0$ and denote by $A_{m_n}$ the set of all virtual permutations $\si$
for which $|R_{m_n}(\si)|\ge\ve$. 
By Corollary~\ref{thm:main_cor1}(b),  
for sufficienlty large $n$
we have (for a suitable $\delta > 0$)
    \[
        \Prob\,(A_{m_n}) 
        < {m_n}^{-1/8 + \delta}
        < n^{-10/9}.
    \]
It follows that the series 
$\sum_{n=1}^\infty \Prob\,(A_{m_n})$ converges. Then, by the Borel--Cantelli lemma,
with probability~$1$ only finitely many of the events $A_{m_n}$ occur, which means that, with probability $1$, for all sufficiently large $n$ we have $|R_{m_n}(\si)|<\ve$. Thus, 
$R_{m_n}(\si)\to0$ as $n\to\infty$
for a.e.\ $\si$.
\end{proof}

\begin{conjecture}\label{conj:virtual}
    For almost every virtual permutation $\si =(\si_1,\si_2,\ldots) \in \S_1\times\S_2\times\ldots$ (with respect to the Haar measure) and for almost every $\nu$ (with respect to the Plancherel measure), 
    \[
        \langle \psi^{\la(\si_n)},\chi^{\nu_n} \rangle
        \sim \frac{f^{\nu_n}}{z_{\la(\si_n)}} 
        \quad \text{as} \quad n \to \infty .
    \]
\end{conjecture}



    

\subsection{Conjugacy characters}\label{sec:conjugacy}

Replacing, in the construction of higher Lie characters described in~\eqref{eq:wreathproduct}, the primitive irreducible character 
$\zeta_k$ of $\bbz_k$ by any character of the form $\zeta_k^r$ for $r \in \bbz$,
we obtain the more general character 
\[
    \psi^\la_r:=\Ind_{Z_\la}^{\S_n}\left(\id_{m_1}[\zeta_1^r]\otimes\id_{m_2}[\zeta_2^r]\otimes\ldots\otimes\id_{m_n}[\zeta_n^r]\right),
\]
where $\id_{m_i}[\zeta_i^r]$ is the wreath product of $\zeta_i^r$ by the trivial character $\id_{m_i}$ of $\S_{m_i}$. Clearly, an analog of~\eqref{eq:hlcformula} also holds:
\[
   \psi^\la_r =\Ind_{\S_{m_1}[\S_1]\times\S_{m_2}[\S_2]\times\ldots}^{\S_n}
\left(\id_{m_1}[\psi_r^{(1)}]\otimes\id_{m_2}[\psi_r^{(2)}]\otimes\ldots\right),
\]
where $\S_{m_i}[\S_i]$ is the wreath product of $\S_i$ by $\S_{m_i}$ and $\id_{m_i}[\psi_r^{(i)}]$ is the wreath product of the character $\psi_r^{(i)}$ of $\S_i$ 
(corresponding to the one-row diagram $(i)$)
by the trivial character $\id_{m_i}$ of $\S_{m_i}$.

\begin{theorem}\label{thm:conjugacy}
All 
the previous results in this paper, in particular Theorems~\ref{main:rectangular} and~\ref{thm:main},
are valid for $\psi_r^\la$, for arbitrary $r$, instead of the higher Lie character~$\psi^\la = \psi_1^\la$.
\end{theorem}
\begin{proof}
It is easy to observe that all the proofs in the previous sections use only the following two facts: 
(a) Theorem~\ref{prop:Lie} about the asymptotics of the ordinary Lie character $\psi^{(n)}$ corresponding to the one-row diagram $\la=(n)$; 
(b) the wreath product construction~\eqref{eq:wreathproduct} of a higher Lie character $\psi^\la$, for an arbitrary $\la$, from ordinary Lie characters. By construction, the structure~(b) is the same for $\psi_r^\la$. 
The proof of Theorem~\ref{prop:Lie}, in turn, uses only the Kra\'skiewicz--Weyman formula (Theorem~\ref{thm:KW})
\[
    \langle\psi^{(n)},\chi^\la\rangle
    =\#\{T\in\operatorname{SYT}(\la)\,\colon \operatorname{maj}T\equiv 1 \!\!\!\pmod n\} =: a_{n,1}
\]
    and Swanson's result (Theorem~\ref{thm:Swanson}) about the asymptotics of $a_{n,1}$. However, both results hold in a more general situation. Namely,  according to Kra\'skiewicz--Weyman~\cite[second corollary to Theorem 1]{KraskiewiczWeyman},
\[
    \langle\psi^{(n)}_r,\chi^\la\rangle=
    \#\{T\in\operatorname{SYT}(\la)\,\colon \operatorname{maj}T\equiv r \!\!\!\pmod n\} =: a_{n,r}
\]
holds for any $r$, 
and Swanson's result is already stated for any $r$. Thus, Theorem~\ref{prop:Lie} holds for $\psi_r^\la$ as well, which implies, according to the observation above, that all the previous results in this paper also hold for the character $\psi_r^\la$.
\end{proof}

In particular, the case $r=0$ corresponds to the conjugacy characters of $\S_n$. Namely, consider the linear
representation $\pi_C$ of $\S_n$ on the group algebra $\CC[\S_n]$ defined by the action of $\S_n$ on itself by conjugation: $\pi_C(g)(h)= ghg^{-1}$.
Clearly, for each $\la\vdash n$, the subspace $H_\la\subset\mathbb C[\S_n]$ spanned by the
permutations with cycle structure $\la$ is invariant under $\pi_C$. The {\em conjugacy character} $\phi^\la$ of $\S_n$ is the character of the restriction
of $\pi_C$ to $H_\la$. It is not difficult to see that
$\phi^\la=\psi_0^\la$, i.e.,
\[
    \phi^\la
    = \Ind_{Z_\la}^{\S_n}\id_{Z_\la}
    = \Ind_{Z_\la}^{\S_n}\left( \id_{m_1}[\iota_1] \otimes \id_{m_2}[\iota_2] \otimes \ldots \otimes \id_{m_n}[\iota_n] \right)
\]
where $\iota_k=\zeta_k^0$ is the trivial character of~$\bbz_k$.

\begin{corollary}\label{cor:conjugacy}
   All the previous results in this paper are valid for the conjugacy character $\phi^\la$ instead of the higher Lie character~$\psi^\la$. 
\end{corollary}

\subsection{Characters arising from symmetric group actions on (co)homology}

Characters close to higher Lie characters appear in natural symmetric group actions on certain homology and cohomology groups.

For instance, Stanley~\cite{Stanley1982} proved that the character $\pi_n$ of the action of $\S_n$ on the top component of the cohomology ring of the partition lattice $\Pi_n$,  
which is isomorphic to the cohomology ring of the pure braid group and to that of the Orlik--Solomon algebra of the braid hyperplane arrangement, 
is equal to $\ve_n\psi^{(n)}$, where $\ve_n$ is the sign character of $\S_n$.

\begin{corollary}
    The sequence of characters $\pi_n$ tends to be regular.    
\end{corollary}
\begin{proof}
    Consider the classical involution $\omega$ on symmetric functions.
    By Theorem~\ref{prop:Lie}, for Plancherel-typical $\nu\vdash n$,
    \[
        \langle \ch\,\pi_n,s_\nu \rangle
        =\langle \omega\Lie_n,s_\nu \rangle
        =\langle \Lie_n,\omega s_\nu \rangle
        =\langle \Lie_n,s_{\nu'} \rangle
        \sim \frac{f^{\nu'}}{n}
        =\frac{f^{\nu}}{n}
        \quad \text{as} \quad n \to \infty. \qedhere
    \]
\end{proof}


More generally, one can show that all the previous proofs in this paper remain valid if in~\eqref{eq:hlcformula} we replace the higher Lie character $\psi^{(i)}$ by $\pi^{(i)}:=\varepsilon_i\psi^{(i)}$, and/or replace some of the trivial characters~$\id_{m_i}$ by~$\varepsilon_{m_i}$. In particular, all our results are valid for the characters
\[
    \Ind_{\S_{m_1}[\S_1]\times\S_{m_2}[\S_2]\times\ldots}^{\S_n}
    \left(\id_{m_1}[\pi^{(1)}]\otimes\varepsilon_{m_2}[\pi^{(2)}]\otimes\ldots
    \otimes\id_{m_{2i-1}}[\pi^{(2i-1)}]\otimes\varepsilon_{m_{2i}}[\pi^{(2i)}]\otimes\ldots\right),  
\]
which appear in the description of the actions of $\S_n$ on the Whitney homology of the partition lattice (see~\cite[Theorem~2.14]{Hersh}) and on the cohomology of the configuration space of points in $\bbr^d$ (see~\cite[Theorem~2.7]{Hersh}).

\section{Further questions}\label{sec:questions}

\subsection{Thresholds}

In Theorem~\ref{thm:rectangular}, regarding rectangular diagrams, it is assumed that the height of the rectangle is asymptotically smaller than its width. 
Indeed, there are families of rectangles which do not tend to be regular. Here is an important example. 

\begin{proposition}\label{prop:2xn}
    The sequence of higher Lie characters indexed by rectangular diagrams $\la=(2^k)$ with constant width $m=2$ and height $k\to\infty$ does not tend to be regular.
\end{proposition}
\begin{proof}
    It is easy to see that $\Lie_2=e_2$. By~\eqref{eq:hlcformula}, we thus have 
    \[
        L_{(2^k)} = h_k[\Lie_2] = h_k[e_2].
    \]
    By \cite[Ex.~I.8.6]{Macdonald},
    \[
        h_k[e_2]
        = \sum_{\nu \vdash 2k,\, \text{all columns even}} s_\nu,
    \]
    so $\langle L_{(2^k)},s_\nu\rangle$ is either $0$ or $1$ for every $\nu$, hence \eqref{claim} obviously fails. 
\end{proof}

\begin{problem}
    What is the threshold for the height $k$, as a function of the width $m$, for a sequence of higher Lie characters indexed by rectangular diagrams $(m^k)$ to tend to be regular?
\end{problem}





An analogous problem can be posed for hook shapes.

\begin{problem}
    What is the threshold for the leg length $k$, as a function of the total size $n$, for a sequence of higher Lie characters indexed by hook shapes $(n-k,1^k)$ to tend to be regular?
\end{problem}

Corollary~\ref{cor:hooks} provides a lower bound for the threshold. Here is an upper bound.

\begin{proposition}\label{prop:hooks_t}
    For any $\ve > 0$, any sequence of higher Lie characters indexed by hook shapes $(n-k,1^k)$ with $k> (2 + \ve) \sqrt n$ for infinitely many values of $n$ does not tend to be regular.    
\end{proposition}

\begin{proof}
    Noting that $L_{(n-k,1^k)} = L_{(n-k)} L_{(1^k)} = L_{n-k} h_k$, the Pieri  
    rule implies that for any $\nu^{(n)} \vdash n$, $\langle L_{(n-k,1^k)}, s_{\nu^{(n)}} \rangle \ne 0 \Longrightarrow \nu^{(n)}_1 \ge k$. 
    Recalling that $\nu^{(n)}_1 \sim 2\sqrt n$ for a Plancherel-typical sequence $\nu$, one deduces that 
    if $k> (2 + \ve) \sqrt n$ for infinitely many values of $n$, then
    $\langle L_{(n-k,1^k)}, s_{\nu^{(n)}} \rangle = 0$ for infinitely many values of $n$.
    Thus $\psi^{(n-k,1^k)}$ does not tend to be regular.
\end{proof}



\subsection{Conjugacy characters 
in other groups}

An old problem in group representation theory is to determine the irreducible modules and their multiplicities in a conjugacy action of a finite group; see, e.g., \cite{Heide} and references therein.

It is well known that the conjugacy action of the symmetric group on itself tends to be regular; see, e.g., \cite{AF, R96}. 
This result is refined in Section~\ref{sec:conjugacy} above: 
by Theorem~\ref{thm:conjugacy}, the conjugacy action of $S_n$ on a random conjugacy class tends to be regular. 
It is natural to ask whether a similar phenomenon holds for finite groups of Lie type.  

Of special interest are the special linear groups.   
For an element $x\in G=SL_n(\FF_q)$, consider the action of $G$ by conjugation on the conjugacy class  of $x$. 
      
\begin{problem} 
    For which sequences of elements does this character tend 
    to be regular, 
    as $n\rightarrow \infty$ or $q\rightarrow \infty$ ?
\end{problem}

In particular, for a fixed finite field $\FF_q$, does an analogue of Theorem~\ref{thm:main1} hold for a sequence of uniformly random elements $x_n \in SL_n(\FF_q)$ ?

\subsection{Descent set distribution}

Definition~\ref{def:regularity} uses inner products with irreducible characters (equivalently, Schur functions). 
One can ask what happens if these are replaced, for example, by {\em ribbon Schur functions}, indexed by compositions.
Their inner products with higher Lie characters enumerate, by the Gessel--Reutenauer theorem~\cite[Theorem 2.1]{GR93}, permutations with a given descent set and a given cycle structure. 

\begin{problem}
    Do results analogous to the ones presented in this paper hold for ribbon Schur functions?
\end{problem}

    Asymptotic estimates on the descent set distribution of permutations in random conjugacy classes may follow. 







\end{document}